\documentclass[a4paper]{article}

\usepackage[english]{babel}

\usepackage{cite}

\usepackage[utf8]{inputenc}
\usepackage[T1]{fontenc}
\usepackage{lmodern}

\usepackage{amssymb,amsmath,amsthm}

\usepackage{todonotes}
\usepackage[shortlabels]{enumitem}
\usepackage{mathtools,bbm}
\usepackage[normalem]{ulem}
\usepackage{cancel}
\usepackage{verbatim}

\usepackage{a4}
\usepackage{multirow}
\usepackage[footnotesize,bf,centerlast]{caption}
\usepackage[small,euler-digits,icomma,OT1,T1]{eulervm}

\usepackage{hyperref}

\usepackage{xcolor,colortbl}
\definecolor{Gray}{gray}{0.80}
\definecolor{LightGray}{gray}{0.90}

\newcommand{\cA}{\mathcal{A}}

\newcommand{\cC}{\mathcal{C}}
\newcommand{\cD}{\mathcal{D}}

\newcommand{\cH}{\mathcal{H}}

\newcommand{\cL}{\mathcal{L}}

\newcommand{\cP}{\mathcal{P}}

\newcommand{\cX}{\mathcal{X}}

\newcommand{\bE}{\mathbb{E}}

\newcommand{\bN}{\mathbb{N}}

\newcommand{\bQ}{\mathbb{Q}}
\newcommand{\bR}{\mathbb{R}}

\newcommand{\bfH}{\mathbf{H}}

\newcommand{\bfR}{\mathbf{R}}

\newcommand{\bfV}{\mathbf{V}}

\newcommand{\PR}{\mathbb{P}}
\newcommand{\bONE}{\mathbbm{1}}

\newcommand{\dd}{ \mathrm{d}}

\DeclareMathOperator*{\LIM}{LIM}

\renewcommand{\epsilon}{\varepsilon}

\newcommand{\vn}[1]{\left| \! \left| #1\right| \! \right|}

\numberwithin{equation}{section}

\newtheorem{theorem}{Theorem}[section]
\newtheorem{lemma}[theorem]{Lemma}
\newtheorem{proposition}[theorem]{Proposition}

\theoremstyle{definition}
\newtheorem{definition}[theorem]{Definition}

\newtheorem{remark}[theorem]{Remark}

\newtheorem{assumption}[theorem]{Assumption}
\newtheorem{example}[theorem]{Example}

\title{Flux large deviations of weakly interacting jump processes via well-posedness of an associated Hamilton-Jacobi equation}

\author{Richard C. Kraaij\thanks{Delft Institute of Applied Mathematics, Delft University of Technology, Van Mourik Broekmanweg 6, 2628 XE Delft, The Netherlands. \emph{E-mail address}: r.c.kraaij@tudelft.nl}}
\date{\today}

\begin{document}

\maketitle

\begin{abstract}
	We establish uniqueness for a class of first-order Hamilton-Jacobi equations with Hamiltonians that arise from the large deviations of the empirical measure and empirical flux pair of weakly interacting Markov jump processes. As a corollary we obtain such a large deviation principle in the context of weakly interacting processes with time-periodic rates in which the period-length converges to $0$. \\
	

\end{abstract}

%


\section{Introduction}

Systems of interacting Markov jump processes appear in various contexts, e.g. in statistical physics, kinetic theory, queuing systems and communication networks. A first natural goal is to understand the limiting behaviour of appropriate observables as the number of components goes to infinity. An extension of this question is that of a large deviation principle, see e.g. \cite{Com87,Le95,DPdH96,ShWe05,BoSu12,DuRaWu16,Kr16b,Re18} and references therein.

We will consider the context of time-inhomogeneous interacting jump processes 
\begin{equation} \label{eqn:intro_weakly_interacting_processes}
	(X_{n,1}(t),\dots,X_{n,n}(t))_{t \geq 0}
\end{equation} 
on a finite-state space $\{1,\dots,q\}$. We assume that the processes are fully exchangeable, jump one-by-one, and interact weakly: their jump rates depend on their empirical measure $\mu_n(t) := n^{-1} \sum_{i=1}^n \delta_{X_{n,i}(t)}$.  We will study the large deviation behaviour of the trajectory of empirical measures $t \mapsto \mu_n(t)$ as $n$ gets large.

We assume that that the interaction has the following properties.
\begin{enumerate}[(1)]
	\item \label{item:intro_weak_interaction}The interaction is \textit{weak}: each of the $n$ process in \eqref{eqn:intro_weakly_interacting_processes} jumps over the bond $(a,b) \in \Gamma := \left\{(a,b) \in \{1,\dots,q\}^2 \, \middle| \, a \neq b\right\}$ with rate $r_n(t,a,b,\mu_n)$,
	\item \label{item:intro_time_periodic} The jump rates are \textit{time-periodic with decreasing period size}. That is, there is a constant $T_0 > 0$ and a sequence of constants $\gamma_n \rightarrow \infty$ such that 
	\begin{equation*}
		r_n(t + \gamma_n^{-1}T_0,a,b,\mu_n) = r_n(t,a,b,\mu_n)
	\end{equation*}
	for all $t \geq 0$, $\mu_n$ and $(a,b) \in \Gamma$. 
	\item \label{item:intro_convergence} The rates are \textit{converging}: there is a kernel $r(t,a,b,\mu)$ such that 
	\begin{equation*}
		\lim_{n \rightarrow \infty} \sup_{t \leq T_0} \sup_{(a,b) \in \Gamma, \mu \in \cP_n(\{1,\dots,n\})} \left|r_n(\gamma_n^{-1}t, a,b,\mu) - r(t,a,b,\mu) \right| = 0,
	\end{equation*}
	where $\cP_n(\{1,\dots,q\})$ is the set of measures of the form $n^{-1} \sum_{i = 1}^n \delta_{x_i}$ for $x_1,\dots,x_n \in \{1,\dots,q\}$.
	\item \label{item:intro_Lipschitz} The rates are \textit{Lipschitz}: there is some $C > 0$ such that 
	\begin{equation*}
		\sup_n \sup_{t \leq T_0} \sup_{\mu,\nu \in \cP_n(\{1,\dots,q\})} \sum_{(a,b) \in \Gamma} \left|r_n(t,a,b,\mu) - r_n(t,a,b,\nu)\right| \leq C\left|\mu - \nu \right|.
	\end{equation*}
\end{enumerate}
The periodicity on a time-interval that is decreasing in length has the effect that the interacting particle system undergoes an effective averaging effect and this will be seen in the final large deviation result. Note that the $\gamma_n$ do not model a speed-up of the process, but rather model an external factor which lives on a faster time-scale.

\smallskip

Recent works on path-space large deviations by \cite{Re18,PaRe19} and works in mathematical physics \cite{MNW08,BeDSGaJLLa02}, or \cite{BeDSGaJLLa06,BeFaGa15,BeChFaGa18} on the study of hydrodynamic limits or long-time (Donsker-Varadhan) large deviations, have shown that studying the process of the empirical measures together with the empirical fluxes simplifies proofs and gives greater insight in the large deviation principles. We will follow these insights and study the empirical measures of the processes in \eqref{eqn:intro_weakly_interacting_processes} in combination with their empirical fluxes.

This paper can thus be seen as a natural continuation of \cite{PaRe19,Kr16b,DuRaWu16}. The papers \cite{DuRaWu16,PaRe19} are more general in the sense that they consider contexts where multiple processes can jump at the same time. If we restrict their results to the context where only a single process jumps we extend the three papers by including a time averaging effect. In addition, we  extend \cite{Kr16b} by including fluxes, \cite{DuRaWu16} by allowing more general rates and include fluxes, and \cite{PaRe19} by including more general rates. Finally, we establish the large deviation principle by using a non-standard technique using the machinery of Hamilton-Jacobi equations introduced by \cite{FK06}. We give a more elaborate comparison after the statement after the introduction of our main results.

\smallskip

Consider the processes \eqref{eqn:intro_weakly_interacting_processes} and denote by $W_{n,i}(t)$ the number of jumps made by $X_{n,i}(t)$ up to time $t$ across each directed edge $(a,b) \in \Gamma$.  We will establish the large deviation principle for the trajectory of the empirical measure-flux pair
\begin{equation} \label{eqn:intro_pair_of_processes}
	t \mapsto Z_n(t) := \left(\frac{1}{n} \sum_{i=1}^n \delta_{X_{n,i}(t)},\frac{1}{n} \sum_{i=1}^n W_{n,i}(t)\right)
\end{equation}
on the Skorokhod space of trajectories in $E := \cP(\{1,\dots,q\}) \times (\bR^+)^{\Gamma}$. The rate function is given in Lagrangian form: 
\begin{equation*}
	I(\mu,w) := \begin{cases}
		I_0(\mu(0),0) + \int_0^\infty \cL((\mu(s),w(s),(\dot{\mu},\dot{w}(s))) \dd s, & \text{if } (\mu,w) \in \cA\cC, \\
		\infty & \text{otherwise},
	\end{cases}
\end{equation*}
where $\cA\cC$ is an appropriate space of absolutely continuous trajectories in $E$. The Lagrangian is given as a sum over relative entropies $S(z \, | \, v)  := z \log \frac{z}{v} - z + v$:
\begin{multline}
	\cL((\mu(s),w(s),(\dot{\mu},\dot{w}(s))) \label{eqn:intro_lagrangian} \\
	:= \begin{cases}
		\sum_{(a,b) \in \Gamma} S\left( \dot{w}_{(a,b)} \, | \, \mu(a) \overline{r}(a,b,\mu)\right) & \text{if } \forall \, a: \dot{\mu}(a) =\sum_b \dot{w}_{(a,b)} - \dot{w}_{(b,a)}, \\
		\infty & \text{otherwise}.
	\end{cases}  
\end{multline}
The kernel $\overline{r}$ denotes the outcome of the averaging principle from \ref{item:intro_time_periodic} and \ref{item:intro_convergence}:
\begin{equation*}
	\overline{r}(a,b,\mu) := \frac{1}{T_0} \int_0^{T_0} r(t,a,b,\mu) \dd t.
\end{equation*}

The key step in the proof of the large deviation result in this paper, and in addition our second main result, is the establishment of the comparison principle (implying uniqueness of viscosity solutions) to a collection of associated Hamilton-Jacobi equations $f - \lambda Hf = h$, for $\lambda > 0$ and $h \in C_b(E)$. The operator $H$ in this equation is given by $Hf(\mu,w) = \cH((\mu,w),\nabla f(\mu,w))$ where $\cH$ is the Legendre transform of $\cL$ from \eqref{eqn:intro_lagrangian}. Its explicit representation is given by 
\begin{multline} \label{eqn:intro_Hamiltonian_difficult}
	\cH((\mu,w),p) = \sum_{(a,b) \in \Gamma} \mu(a) \overline{r}(a,b,\mu)\left[\exp\left\{p_b - p _a + p_{(a,b)} \right\} - 1 \right], \\
	(\mu,w) \in \cP(\{1,\dots,q\}) \times (\bR^+)^\Gamma, p \in \bR^{q} \times \bR^{\Gamma}.
\end{multline} 
Due to the terms of the type $\mu(a) r(a,b,\mu)\left[e^{p_b - p_a + p_{(a,b)}} - 1 \right]$ the Hamiltonian is neither Lipschitz nor uniformly coercive in $p$. This implies that our Hamilton-Jacobi equation can not be treated using `standard' methods for first-order equations, see \cite{FlSo06,BaCD97,CIL92} and references therein. Instead, our method improves upon the method of \cite{Kr16b} which was designed for the Hamiltonian of weakly interacting jump processes without taking into account the fluxes. The novelty of the proof of the comparison principle, compared to \cite{Kr16b}, is based on a novel `two stage' penalization procedure, which potentially can be used to treat other types of `non-standard' first-order Hamilton-Jacobi equations, see Sections \ref{section:general_method_for_comparison_principle} and \ref{section:comparison_explicit}.

We stress that the verification of the comparison principle is of interest beyond the large deviation statement that is proven in this paper. First of all, the comparison principle can find other applications in the field of control theory or mean-field games. Secondly, an extension of the comparison principle in this paper by the bootstrap principle introduced in \cite{KrSc19} leads to comparison principles for more elaborate Hamilton-Jacobi (-Bellman) equations. In turn these boosted comparison principles can be used for new large deviation principles, as carried out in the forthcoming work \cite{KrSchl20} in the context of more general slow-fast systems.

\smallskip

We next compare our large deviation result to results in the literature. 

Large deviations for weakly interacting jump processes have been studied in the past, see e.g. \cite{Com87,Le95,DPdH96} in contexts with spatial structure or random fields. The methods of proof were based on direct evaluation of the asymptotics or tilting arguments based on Sanov's Theorem, Varadhan's lemma and the contraction principle.

\smallskip

More recent papers in the context of non-spatial processes have focused on different methods of proof \cite{Kr16b,DuRaWu16}, or have included fluxes \cite{Re18,PaRe19}. Of these four papers, two \cite{Kr16b,Re18} still focus on processes with transitions of the type where one particle moves its state, whereas other two papers \cite{DuRaWu16,PaRe19} allow for more general transitions, e.g. allowing more particles to change their state at a single time or consider mass-action kinetics.  

As a first remark, this paper includes an averaging effect for path-space large deviations. If we restrict ourselves to the time-homogenous case, we can compare our large deviation principle to those of \cite{Kr16b,Re18,DuRaWu16,PaRe19}. We focus our comparison to \cite{DuRaWu16,PaRe19}, as this paper supersedes \cite{Kr16b} by the inclusion of fluxes, and \cite{PaRe19} supersedes \cite{Re18} by generalizing the single-jump setting as well as letting go of the independence assumption.

In \cite{DuRaWu16}, the authors work in the context without fluxes. The proof of the large deviation principle is based on a variational expression for the Poisson random measure, of which it is established that the expression converges as $n \rightarrow \infty$. An approximation argument based on ergodicity is used to reduce the proof of the lower bound to trajectories that lie in the interior of the simplex of probability measures. It is also assumed that the law of large numbers limit pushes the empirical measure into the interior of the simplex. If restricted to the context of single-jumps only, this paper covers more cases as assumptions of the type resembling these two final conditions of \cite{DuRaWu16} are absent from this paper.

In \cite{PaRe19}, following \cite{Re18}, the empirical measure is combined with the empirical fluxes. The inclusion of the fluxes allows for a clear and direct change of measure argument leading in a straightforward way to the Lagrangrian in terms of a sum over appropriate relative entropies. In the context of single jumps, our result extends that of \cite{PaRe19}. Two key assumptions in \cite{PaRe19} are Assumption 2.2 (v) and (vi). The first of these two conditions is naturally reflected by the assumption that the limit of the jump rates form a \textit{proper kernel} as in Assumption \ref{assumption:jump_rates} \ref{item:results_rate_zero_or_positive}. It should be noted that (v) of \cite{PaRe19} is more restrictive and excludes for example Glauber type interactions like in Example \ref{example:Gibbs_dynamics}. In addition, this paper does not assume an analogue of \cite[Assumption 2.2 (vi)]{PaRe19}.

\smallskip

We thus see that the proof via the comparison principle in the context of systems with single jumps yields the most general results, and with additional work would allow for a generalization to the context where the rates are non-Lipschitz as in \cite{Kr16b}. The proof of the comparison principle, however, uses a technique that is very much geared towards Hamiltonians of the type \eqref{eqn:intro_Hamiltonian_difficult} and can not directly be adapted to the more general setting of processes with multiple simultaneous jumps of \cite{DuRaWu16,PaRe19}.  More remarks on these restrictions are given in Section \ref{section:comparison_explicit}.

\smallskip

The paper is organized as follows. We start in Section \ref{section:preliminaries} with basic definitions, including those of viscosity solutions of Hamilton-Jacobi equations, the comparison principle, the martingale problem, and the large deviation principle.  In Section \ref{section:main_results} we state our main results: the comparison principle and the large deviation principle. 

In Section \ref{section:LDP_via_HJ} we give the key results that reduce the proof of the large deviation principle to the comparison principle. We then prove the comparison principle in Section \ref{section:proofs_comparison} and we follow with the verification of the remaining assumptions for the results of Section \ref{section:LDP_via_HJ} in Section \ref{section:proof_ldp_explicit_other_steps}. In Appendix \ref{appendix:viscosity_solutions}, we collect some results for the literature that are essential for the proof of the comparison principle. Their inclusion makes the paper as self-contained as possible.

\section{Preliminaries} \label{section:preliminaries}

Let $E$ be a Polish space. We denote by $\cP(E)$ the space of Borel probability measures on $E$. By $\cP_n(E)$ we denote the subset of measures that have the form $n^{-1} \sum_{i=1}^n \delta_{x_i}$ for some collection $\{x_i\}_{i=1}^n \subseteq E$.

We denote by $D_E(\bR^+)$ the space of paths $\gamma : \bR^+ \rightarrow E$ that are right continuous and have left limits. We endow $D_E(\bR^+)$ with the Skorokhod topology, cf. \cite[Section 3.5]{EK86}. An important property is that under this topology $D_E(\bR^+)$ is Polish if $E$ is Polish.

We denote by $C(E)$ and $C_b(E)$ the spaces of continuous and bounded continuous functions on $E$. For $d \in \bN \setminus \{0\}$ and $k \in \bN$ let $C_b^k(\bR^d)$ be the space of functions that have $k$ continuous and bounded derivatives. By $C_b^\infty(\bR^d)$ we denote the space of functions with bounded continuous derivatives of all orders. 

Now consider a subset $E \subseteq \bR^d$ that is a Polish space and that is contained in the $\bR^d$ closure of its $\bR^d$ interior. We denote by $C^k_b(E)$, $C^\infty_b(E)$ the spaces of functions that have an extension to $C^k_b(\bR^d)$ and $C^\infty_b(\bR^d)$ respectively. Finally, denote by $C_c^k(E)$ and $C_c^\infty(E)$ the subsets that have compact support in $E$. Note that the derivative of a continuously differentiable function on $E$ is determined by the values of the function on $E$ by our assumption on $E$.

Finally, we introduce the space $\cA\cC(E)$ of absolutely continuous paths in $E$. A curve $\gamma: [0,T] \to E$ is absolutely continuous if there exists a function $g \in L^1([0,T],\bR^d)$ such that for $t \in [0,T]$ we have $\gamma(t) = \gamma(0) + \int_0^t g(s) \dd s$. We write $g = \dot{\gamma}$.\\
A curve $\gamma: \bR^+ \to E$ is absolutely continuous, i.e. $\gamma \in \cA\cC(E)$, if the restriction to $[0,T]$ is absolutely continuous for every $T \geq 0$.

\subsection{Large deviations} \label{section:preliminaries_LDP}

Let $\cX$ be a Polish space. Later we will use both $\cX = D_E(\bR^+)$ and $\cX = E$.

\begin{definition}
	Let $\{X_n\}_{n \geq 1}$ be a sequence of random variables on $\mathcal{X}$. Furthermore, consider a function $I : \mathcal{X} \rightarrow [0,\infty]$. We say that
	\begin{itemize}
		\item  
		the function $I$ is a \textit{good rate-function} if the set $\{x \, | \, I(x) \leq c\}$ is compact for every $c \geq 0$;
		%
		\item 
		the sequence $\{X_n\}_{n\geq 1}$ satisfies the \textit{large deviation principle} and good rate-function $I$ if for every closed set $A \subseteq \mathcal{X}$, we have 
		\begin{equation*}
			\limsup_{n \rightarrow \infty} \, \frac{1}{n} \log \PR[X_n \in A] \leq - \inf_{x \in A} I(x),
		\end{equation*}
		and, for every open set $U \subseteq \mathcal{X}$, 
		\begin{equation*}
			\liminf_{n \rightarrow \infty} \, \frac{1}{n} \log \PR[X_n \in U] \geq - \inf_{x \in U} I(x).
		\end{equation*}
	\end{itemize}
\end{definition}

\subsection{The martingale problem} \label{section:martingale_problem}

One effective way of defining a Markov process on $E$ is by using its infinitesimal generator, see e.g. \cite{EK86}. One of the instances of this idea is that of solving the martingale problem. 

We introduce the martingale problem for time-inhomogeneous processes. Note that this is a straightforward extension from the time-homogeneous case via the inclusion of time in the state-space, see for example Section 4.7.A in \cite{EK86} or Proposition II.5.7 in \cite{Pe02}.

Let $A : \cD(A) \subseteq C_b(E) \rightarrow C_b(\bR^+ \times E)$ be a linear operator. For each time $s \in \bR^+$, we denote by $A[s] : \cD(A) \subseteq C_b(E) \rightarrow C_b(E)$ the linear operator obtained by fixing $s$. $A[s]$ can be interpreted as the generator at time $s$. In addition, we construct out of the operators $A[s]$ an operator $\vec{A}$ on $C_b(\bR^+ \times E)$:
\begin{itemize}
	\item $\cD(\vec{A})$ satisfies
	\begin{equation*}
		\cD(\vec{A}) \subseteq \left\{f \in C_b(\bR^+ \times E) \, \middle| \, \forall \, x \in E \colon \, f(\cdot,x) \in C^1_b(\bR^+), \, \forall s \in \bR^+ \colon \, f(s,\cdot) \in \cD(A) \right\},
	\end{equation*}
	\item for $f \in \cD(\vec{A})$ we have $\vec{A}f(s,x) = \partial_s f(s,x) + (A[s]f(s,\cdot))(x)$.
\end{itemize}


\begin{definition}
	Let $\mu \in \cP(E)$. We say that the process $t \mapsto X(t)$ on $D_E(\bR^+)$ solves \textit{the (time-inhomogeneous) martingale problem} for $(\vec{A},\mu)$ if for all $f \in \cD(\vec{A})$ the process
	\begin{align*}
		M_f(t) & := f(t,X(t)) - f(0,X(0))  - \int_0^t  \vec{A} f(s,X(s)) \dd s \\
		& = f(t,X(t)) -  f(0,X(0)) - \int_0^t \partial_s f(s,X(s)) +  (A[s]f(s,\cdot))(X(s)) \dd s 
	\end{align*}
	is a martingale and if the projection of $\PR$ on the time $0$ coordinate equals $\mu$. 
	
	By slight abuse of notation, we will also say that the measure of the process $t \mapsto (t,X(t))$ solves the martingale for $\vec{A}$.
\end{definition}


\subsection{Viscosity solutions to Hamilton-Jacobi equations} \label{section:preliminaries_viscosity_solutions}

\begin{definition}[Viscosity solutions]
	Let $H : \cD(H) \subseteq C_b(E) \rightarrow C_b(E)$, $\lambda > 0$ and $h \in C_b(E)$. Consider the Hamilton-Jacobi equation
	\begin{equation}
		f - \lambda H f = h. \label{eqn:differential_equation} 
	\end{equation}
	We say that $u$ is a \textit{(viscosity) subsolution} of equation \eqref{eqn:differential_equation} if $u$ is bounded, upper semi-continuous and if, for every $f \in \cD(H)$ there exists a sequence $x_n \in E$ such that
	\begin{gather*}
		\lim_{n \uparrow \infty} u(x_n) - f(x_n)  = \sup_x u(x) - f(x), \\
		\lim_{n \uparrow \infty} u(x_n) - \lambda H f(x_n) - h(x_n) \leq 0.
	\end{gather*}
	We say that $v$ is a \textit{(viscosity) supersolution} of equation \eqref{eqn:differential_equation} if $v$ is bounded, lower semi-continuous and if, for every $f \in \cD(H)$ there exists a sequence $x_n \in E$ such that
	\begin{gather*}
		\lim_{n \uparrow \infty} v(x_n) - f(x_n)  = \inf_x v(x) - f(x), \\
		\lim_{n \uparrow \infty} v(x_n) - \lambda Hf(x_n) - h(x_n) \geq 0.
	\end{gather*}
	We say that $u$ is a \textit{(viscosity) solution} of equation \eqref{eqn:differential_equation} if it is both a subsolution and a supersolution to \eqref{eqn:differential_equation}.
	
	We say that \eqref{eqn:differential_equation} satisfies the \textit{comparison principle} if for every subsolution $u$ and supersolution $v$ to \eqref{eqn:differential_equation}, we have $u \leq v$.
\end{definition}

\begin{remark}
	The comparison principle implies uniqueness of viscosity solutions. Suppose that $u$ and $v$ are both viscosity solutions, then the comparison principle yields that $u \leq v$ and $v \leq u$, implying that $u = v$.
\end{remark}

\begin{remark} \label{remark:existence of optimizers}
	Consider the definition of subsolutions. Suppose that the testfunction $f \in \cD(H)$ has compact sublevel sets, then instead of working with a sequence $x_n$, we can pick a $x_0$ such that
	\begin{gather*}
		u(x_0) - f(x_0)  = \sup_x u(x) - f(x), \\
		u(x_0) - \lambda H f(x_0) - h(x_0) \leq 0.
	\end{gather*}
	A similar simplification holds in the case of supersolutions.
\end{remark}

\section{Main results} \label{section:main_results}


In this section we give our two main results: the large deviation principle and the comparison principle. We give a short recap of some of the definitions informally given in the introduction. Let $\{1,\dots,q\}$, $q \in \bN\setminus\{0,1\}$ be some finite set. Write $\Gamma = \left\{(a,b) \in \{1,\dots,q\}^2 \, \middle| \, a \neq b \right\}$ for the directed edge-set in $\{1,\dots,q\}$. Let $E = \cP(\{1,\dots,q\}) \times (\bR^+)^\Gamma$ be the Polish space of probability measures on $\{1,\dots,q\}$ combined with a space in which we can keep track of the fluxes over the directed bonds in $\Gamma$.

For $n$ points $\vec{x} = (x_1,\dots,x_n) \in \{1,\dots,q\}$ denote by $\mu_n[\vec{x}]$ the empirical measure $\mu_n[\vec{x}] = n^{-1} \sum_{i=1}^n \delta_{x_i}$.

\smallskip

We consider a collection of weakly-interacting jump processes 
\begin{equation} \label{eqn:results_weakly_interacting_processes}
	\vec{X}_n(t) = (X_{n,1}(t),\dots,X_{n,n}(t))_{t \geq 0}
\end{equation}
on the space $\{1,\dots,q\}$ and $\mu_n(t) := \mu_n[\vec{X}_n(t)]$ the empirical measure of the process at time $t$. For any given $n$, we will assume that each of the $n$ processes, if at state $a$, jumps to state $b$ with rate
\begin{equation*}
	r_n(t,a,b,\mu_n(t)),
\end{equation*}
i.e. the processes interact weakly through their empirical measure.

We are interested in the large deviation behaviour of the trajectory $t \mapsto \mu_n(t)$ on the space $\cP(\{1,\dots,q\})$. Following \cite{Re18,PaRe19}, it turns out that a description of the large deviation principle simplifies if we take into account also the fluxes across the bonds in $\Gamma$. Therefore, denote by
\begin{equation*}
	t \mapsto W_{n,i}(t) \in \bN^\Gamma
\end{equation*}
the process that satisfies
\begin{equation*}
	W_{n,i}(t)(a,b) := \#\left\{s \leq t \, \middle| \, \left(X_{n,i}(s-), X_{n,i}(s)\right) = (a,b) \right\}.
\end{equation*}

Our first result establishes the large deviation principle for the pair of processes
\begin{equation} \label{eqn:main_pair_of_processes}
	t \mapsto Z_n(t) := \left(\frac{1}{n} \sum_{i=1}^n \delta_{X_{n,i}(t)},\frac{1}{n} \sum_{i=1}^n W_{n,i}(t)\right),
\end{equation}
on the set $D_E(\bR^+)$.

In Section \ref{section:LDP_periodic}, we state our large deviation principle. In Section \ref{section:examples} we give an example in the context of Glauber dynamics.  We end our section of main results in Section \ref{section:comparison_principle} with the uniqueness result for the associated Hamilton-Jacobi equations.

\subsection{Flux large deviations for time-periodic jump rates} \label{section:LDP_periodic}

\begin{assumption} \label{assumption:jump_rates}
	The rates $r_n(t,a,b,\mu)$ at which each of the processes in \eqref{eqn:results_weakly_interacting_processes} jumps over the bond $(a,b) \in \Gamma$ at time $t$ while the empirical measure is given by $\mu$ satisfies the following properties.
	\begin{enumerate}[(a)]
		\item \label{item:results_time_periodic} The jump rates are \textit{time-periodic with decreasing period size}. That is, there is a constant $T_0 > 0$ and a sequence of constants $\gamma_n \rightarrow \infty$ such that $r_n(t + \gamma_n^{-1}T_0,a,b,\mu_n) = r_n(t,a,b,\mu_n)$ for all $t \geq 0$, $\mu_n$ and $(a,b) \in \Gamma$. 
		\item \label{item:results_convergence} The rates are \textit{converging}: there is a kernel $r(t,a,b,\mu)$ such that 
		\begin{equation*}
			\lim_{n \rightarrow \infty} \sup_{t \leq T_0} \sup_{(a,b) \in \Gamma, \mu \in \cP_n(\{1,\dots,n\})} \left|r_n(\gamma_n^{-1}t, a,b,\mu) - r(t,a,b,\mu) \right| = 0,
		\end{equation*}
		where $\cP_n(\{1,\dots,q\})$ is the set of measures of the form $n^{-1} \sum_{i = 1}^n \delta_{x_i}$ for $x_1,\dots,x_n \in \{1,\dots,q\}$.
		\item \label{item:results_rate_zero_or_positive} \textit{The averaged kernel is `proper'.} Denote by $\overline{r}$ the kernel 
		\begin{equation*}
			\overline{r}(a,b,\mu) := \frac{1}{T_0} \int_0^{T_0} r(t,a,b,\mu) \dd t.
		\end{equation*}
		For each $(a,b) \in \Gamma$, we have either (i) or (ii):
		\begin{enumerate}[(i)]
			\item $\overline{r}(a,b,\mu) = 0$ for all $\mu$,
			\item $\inf_{\mu} \overline{r}(a,b,\mu) > 0$.
		\end{enumerate}
		\item \label{item:results_Lipschitz} The rates are \textit{Lipschitz}: there is some $C > 0$ such that 
		\begin{equation*}
			\sup_n \sup_{t \leq T_0} \sup_{\mu,\nu \in \cP_n(\{1,\dots,q\})} \sum_{(a,b) \in \Gamma} \left|r_n(t,a,b,\mu) - r_n(t,a,b,\nu)\right| \leq C\left|\mu - \nu \right|.
		\end{equation*}
	\end{enumerate}

\end{assumption}

\begin{remark}
	Note that the time-periodic context includes the time-homogeneous contexts. Namely, if the rates do not depend on $t$ than any choice of $T_0 > 0$ and $\gamma_n \rightarrow \infty$ will induce time-periodicity of the type above.
\end{remark}

Our first main result is the large deviation principle for the pair of processes \eqref{eqn:main_pair_of_processes}  in the context of time periodic rates.

\begin{theorem} \label{theorem:ldp_mean_field_jump_process_periodic}
	For each $n$ consider the process of state-flux pairs 
	\begin{equation*}
		((X_{n,1}(t),W_{n,1}(t)),\dots,(X_{n,n}(t),W_{n,n}(t)))_{t \geq 0},
	\end{equation*} 
	where the jump rates $r_n$ satisfy Assumption \ref{assumption:jump_rates}.
	
	\smallskip
	
	Consider the processes $t \mapsto Z_n(t) := \left(\frac{1}{n} \sum_{i=1}^n \delta_{X_{n,i}(t)},\frac{1}{n} \sum_{i=1}^n W_{n,i}(t)\right)$. Suppose that $Z_n(0)$ satisfies a large deviation principle on $E = \cP(\{1,\dots,q\}) \times (\bR^+)^\Gamma$ with good rate function $I_0$,  then $\{Z_n\}_{n \geq 1}$ satisfies the large deviation principle on $D_{E}(\bR^+)$ with good rate function $I$ 
	given by
	\begin{equation*}
		I(\mu,w) = 
		\begin{cases}
			I_0(\mu(0),w(0)) + \int_0^\infty \cL((\mu(s),w(s)),(\dot{\mu}(s),\dot{w}(s))) \dd s & \text{if } (\mu,w) \in \cA\cC(E), \\
			\infty & \text{otherwise},
		\end{cases}
	\end{equation*}
	where $\cL : E \times \bR^{q + |\Gamma|} \rightarrow \bR^+$ is given by
	\begin{equation*}
		\cL\left((\mu,w),(\dot{\mu},\dot{w})\right) = \begin{cases}
			\sum_{(a,b) \in \Gamma} S(\dot{w}_{(a,b)} \, | \, \mu(a) \overline{r}(a,b,\mu))  & \text{if } \forall \, a: \, \dot{\mu}_a = \sum_b \dot{w}_{(b,a)} - \dot{w}_{(a,b)}, \\
			\infty & \text{otherwise},
		\end{cases}
	\end{equation*}
	with 
	\begin{equation*}
		S(z \, | \, v) := \begin{cases}
			v & \text{if } z = 0, \\
			z \log \frac{z}{v} - z + v & \text{if } z \neq 0, v \neq 0, \\
			\infty & \text{if } z \neq 0, v = 0.
		\end{cases}
	\end{equation*}
	
\end{theorem}

Note that the Lagrangian is given in terms of the entropic cost of changing the flux across each bond. Indeed, $S$ is the relative entropy corresponding to a tilt of the intensity of a Poisson measure. 

\smallskip

\begin{remark}
	Theorem \ref{theorem:ldp_mean_field_jump_process_periodic} gives as a corollary the large deviation principle for the trajectory of the empirical measures only. This recovers e.g. the result of \cite{Kr16b} but now in contracted form. The rate function $J$ is given by 
	\begin{multline*}
		J(\mu) = I_0(\mu(0)) + \inf \left\{\int_0^\infty \sum_{(a,b) \in \Gamma} S(\dot{w}_{(a,b)}(s) \, | \, \mu(a) \overline{r}(a,b,\mu(s))) \dd s  \right.\\
		\left.  \, \middle| \, w \in \cA\cC(E), \forall \, a: \, \dot{\mu}(a) = \sum_b \dot{w}_{(b,a)} - \dot{w}_{(a,b)} \right\}
	\end{multline*}
	if $\mu$ is absolutely continuous and $J(\mu) = \infty$ otherwise.
\end{remark}

As a second remark, a comment on the Lipschitz property in Assumption \ref{assumption:jump_rates} \ref{item:results_Lipschitz}.

\begin{remark}
	The Lipschitz assumption can be dropped in the context that the rates are time-homogeneous. 
	
	The uniqueness of solutions of the Hamilton-Jacobi equation in Theorem \ref{theorem:comparison_principle} below does not depend on this statement. Thus, a large deviation principle for the time-homogeneous case without the Lipschitz condition can be obtained by dropping the $F_{f,n}$ or $h_n$ term in Proposition \ref{proposition:convergence_of_operators} and Lemma \ref{lemma:convergence_of_operators_modulo_time_period}. Alternatively, one can adapt the operator convergence proof in \cite{Kr16b}.
	
	Even in the context of  time-inhomogeneous rates, the Lipschitz property is not essential. In a work in progress, \cite{KrSchl20}, this is explored in the more general context of Markov processes with two time-scales. A proof of convergence of operators and how to deal with the viscosity solutions for the limiting operators in this general two-scale context is technically more challenging. We refrain from carrying this out in this context and keep a simpler and independent proof in this paper.
\end{remark}


\subsection{Example: the Curie-Weiss-Potts model} \label{section:examples}

Next, we establish the path-space large deviation principle in the context of the dynamic Curie-Weiss-Potts model.

\begin{example}[Time-dependent potential functions] \label{example:Gibbs_dynamics}
	Let $V : \bR^+ \times \bR^q \rightarrow \bR$ be a twice continuously differentiable function in the second component and such that $V(t +1, \cdot) = V(t,\cdot)$ for all $t \geq 0$. Fix $r_0 : \{1,\dots,q\}\times\{1,\dots,q\} \rightarrow \bR^+$. Finally, let $\gamma_n \rightarrow \infty$ and set
	\begin{equation*}
		r_n(t,a,b,\mu) := r_0(a,b)\exp\left\{- n 2^{-1}\left(V\left(\gamma_n t, \mu - n^{-1}\delta_a + n^{-1} \delta_b\right) - V(\gamma_n t,\mu)\right) \right\}.
	\end{equation*}
	As $n$ goes to infinity, the limiting kernel $r$ becomes
	\begin{equation*}
		r(t,a,b,\mu) := r_0(a,b) \exp\left\{\frac{1}{2} \nabla_a V(t,\mu) - \frac{1}{2} \nabla_b V(t,\mu) \right\},
	\end{equation*}
	so that
	\begin{equation*}
		\overline{r}(a,b,\mu) :=  r_0(a,b) \int_0^1 \exp\left\{\frac{1}{2} \nabla_a V(t,\mu) - \frac{1}{2} \nabla_b V(t,\mu) \right\} \dd t.
	\end{equation*}
\end{example}

\subsection{The comparison principle} \label{section:comparison_principle}

We close our section of main results with the uniqueness of viscosity solutions  for the Hamilton-Jacobi equation.

The motivation for this well-posedness result comes from a connection between large deviation theory of Markov processes and non-linear semigroup theory that by a chain of arguments leads to Theorem  \ref{theorem:ldp_mean_field_jump_process_periodic}. This chain of arguments was first introduced in \cite{FK06} and later reproved in \cite{Kr19,Kr19c}. The key technical steps in this method are collected in Section \ref{section:LDP_via_HJ}.

This reduction, even though at first sight technical, is fully analogous to that of establishing weak convergence of Markov processes and is carried out via the convergence of their a transformed version of their infinitesimal generators. The statement that the martingale problem for the limiting operator is well posed is naturally replaced by uniqueness of solutions for the Hamilton-Jacobi equation (see \cite{CoKu15} for the result that the uniqueness of the martingale problem for a linear operator follows from uniqueness of solutions to the Hamilton-Jacobi equation in terms of this operator).

\smallskip

In the proof of Theorem \ref{theorem:ldp_mean_field_jump_process_periodic}, we see that the natural limiting operator $H$ is  of the form $Hf(\mu,w) = \cH((\mu,w), \nabla f(\mu,w))$ where 
\begin{equation} \label{eqn:Hamiltonian_ldp_in_comparison}
	\cH((\mu,w),p) = \sum_{(a,b) \in \Gamma} \mu(a) \overline{r}(a,b,\mu)\left[\exp\left\{p_b - p _a + p_{(a,b)} \right\} - 1 \right].
\end{equation}

\smallskip

Our second main result is the well-posedness of the Hamilton-Jacobi equation $f - \lambda Hf = h$. We state it separately as it is of use in a context that goes beyond large deviation theory. It also holds in a slightly more general setting than for the Hamiltonian obtained in Theorem \ref{theorem:ldp_mean_field_jump_process_periodic}. We give the proper context.

\begin{definition} \label{definition:limiting_kernel}
	Let $v : \Gamma \times \cP(\{1,\dots,q\}) \rightarrow \bR^+$. We say that $v$ is a \textit{proper kernel} if $v$ is continuous and if for each $(a,b) \in \Gamma$, the map $\mu \mapsto v(a,b,\mu)$ is either identically equal to zero or satisfies the following two properties:
	\begin{enumerate}[(a)]
		\item $v(a,b,\mu) = 0$ if $\mu(a) = 0$ and $v(a,b,\mu) > 0$ for all $\mu$ such that $\mu(a) > 0$. 
		\item There exists a decomposition $v(a,b,\mu) = v_{\dagger}(a,b,\mu(a)) v_{\ddagger}(a,b,\mu)$ such that $v_{\dagger} : \Gamma \times [0,1] \rightarrow \bR^+$ is increasing in the third coordinate (that is, in $\mu(a)$) and such that $v_{\ddagger} : \Gamma \times \cP(\{1,\dots,q\}) \rightarrow \bR^+$ is continuous and satisfies $\inf_{\mu} v_{\ddagger}(a,b,\mu) > 0$.
	\end{enumerate}
\end{definition}

Note that the Hamiltonian in \eqref{eqn:Hamiltonian_ldp_in_comparison} features a proper kernel $v$. Choose $v_\dagger(a,b,\mu) = \mu(a)$ and $v_\ddagger(a,b,\mu) = \overline{r}(a,b,\mu)$ and argue using Assumption \ref{assumption:jump_rates} \ref{item:results_rate_zero_or_positive}.

\begin{theorem} \label{theorem:comparison_principle}
	Consider the Hamiltonian $H : \cD(H) \subseteq C_b( E) \rightarrow C_b(E)$ with domain $C_c^\infty(E) \subseteq \cD(H) \subseteq C_b^1(E)$ satisfying $Hf(\mu,w) = \cH((\mu,w),\nabla f(\mu,w))$ where $\cH : E \times \bR^d \rightarrow \bR$ is given by 
	\begin{equation} \label{eqn:Hamiltonian}
		\cH((\mu,w),p) = \sum_{(a,b) \in \Gamma} v(a,b,\mu)\left[\exp\left\{p_b - p _a + p_{(a,b)} \right\} - 1 \right].
	\end{equation}
	and where $v : \Gamma \times \cP(\{1,\dots,q\}) \rightarrow \bR^+$ is a proper kernel. Then for each $\lambda > 0$ and $h \in C_b(E)$ the comparison principle holds for 
	\begin{equation*}
		f - \lambda H f = h.
	\end{equation*}
\end{theorem}

The proof of the theorem can be found in Section \ref{section:comparison_explicit}.

\begin{remark} \label{remark:dimensionality}
	Note that a natural interpretation of $E$ is that as a subset of $\bR^{q + |\Gamma|}$. However, due to the fact we work with probability measures, we can also interpret $E$ as a subset of $\bR^{q-1 + |\Gamma|}$. The first interpretation is more natural to write down equations or Hamiltonians and we will do so in the subsequent sections stating the main results. Only in the second interpretation our set is a subset of $\bR^{q-1 + |\Gamma|}$  that is contained in the closure of its interior as as will be needed in the proofs of Section \ref{section:proofs_comparison}.
\end{remark}

\section{Large deviations via well-posedness of the Hamilton-Jacobi equation} \label{section:LDP_via_HJ}

The key tool that allows us to establish the path-space large deviation principle is the well-posedness of the Hamilton-Jacobi equation. In this section, we give an outline of this reduction. We base ourselves on the work \cite{FK06} in which this connection was first put on strong footing. We will also refer to \cite{Kr19c} in which a new proof of the functional analytic content of the reduction is given.

\smallskip

The method can best be compared to a similar situation in the context of the weak convergence of Markov processes. There one establishes: tightness, convergence of generators, and the uniqueness (and existence) of solutions of the limiting martingale problem. In the large deviation context, we follow the same strategy \footnote{In fact, the similarity is more than simply an analogy. It was established in \cite{CoKu15} that well-posedness for the Hamilton-Jacobi equation with a linear operator establishes uniqueness for the corresponding martingale problem.}: exponential tightness, convergence of non-linearly transformed generators, and well-posedness of the Hamilton-Jacobi equation for the limiting operator. 

\smallskip

To properly describe the reduction, we start by introducing an appropriate martingale problem for the empirical measures and their fluxes. Afterwards, we introduce the appropriate notion of convergence of operators. After that we give the framework that reduces the large deviation principle to the comparison principle.

\subsection{An appropriate martingale problem}  \label{section:proof_martingale_problem}

We next introduce the martingale problem that corresponds to the evolution of the empirical measure and empirical fluxes of $n$ weakly interacting Markov processes. 

Denote by 
\begin{equation*}
	E_n := \left\{(\mu,w) \in E \, \middle| \, \exists (x_1,\dots,x_n) \in \{1,\dots,q\}, W \in \bN^\Gamma: \mu = \frac{1}{n} \sum_{i=1}^n \delta_{x_i}, w = \frac{1}{n}W\right\}
\end{equation*}
the state space of the process of empirical measures and fluxes when working with $n$ interacting processes.

Following the setting in Section \ref{section:main_results}, we see that at moment that one of the processes $X_{n,i}$ makes a jump from site $a$ to $b$, the empirical measure makes a change from $\mu_n[\vec{X}]$ to $\mu_n[\vec{X}] + \tfrac{1}{n}\left(\delta_b - \delta_a\right)$, whereas the empirical flux $w$ is increased by $n^{-1}$ on the bond $(a,b) \in \Gamma$.

It follows that the corresponding generator for the process 
\begin{equation*}
	t \mapsto Y_n(t) := (t_0 + t,\mu_n(t),w_n(t))
\end{equation*}
on $D_{\bR^+ \times E_n}(\bR^+)$ is given by
\begin{multline*}
	\vec{A}_n f(t,\mu,w) = \partial_t f(t,\mu,x) \\
	+ \sum_{(a,b) \in \Gamma} \mu(a) r_n(t,a,b,\mu) \left[f\left(\mu + \frac{1}{n}(\delta_b - \delta_a), w + \frac{1}{n} \delta_{(a,b)}\right) - f\left(\mu,w\right)  \right].
\end{multline*}

Starting from this generator, we introduce the machinery to establish the large deviation principle.

\subsection{Convergence of operators}

To study the convergence of operators, we first need a notion of convergence of functions on a sequence of spaces.

\begin{definition}
	Let $f_n \in C_b(\bR^+ \times E_n)$ and $f \in C_b(E)$. We say that $\LIM f_n = f$ if 
	\begin{itemize}
		\item $\sup_n \vn{f_n} < \infty$,
		\item for all $T \geq 0$ and $M \geq 0$, we have
		\begin{equation*} 
			\lim_{n \rightarrow \infty} \sup_{\substack{(t,\mu,w) \in \bR^+ \times E_n: \\ t \leq T, |w| \leq M}} \left|f_n(t,\mu,w) - f(\mu,w) \right| = 0.
		\end{equation*}
	\end{itemize}
\end{definition}

The second condition of the definition of $\LIM$ is equivalent to the statement that for all $(\mu,w) \in E$ and any sequence $(t_n,\mu_n,w_n) \in \bR^+ \times E_n$ such that $(\mu_n,w_n) \rightarrow (\mu,w)$ and $\sup_n t_n < \infty$, we have $\lim_n f_n(t_n,\mu_n,w_n) = f(\mu,w)$. Note that for any function $f \in C_b(E)$, we have (by interpreting it as a function on $\bR^+ \times E_n$ by $f(t,x) = f(x)$) the natural statement $\LIM f = f$.

Using the notion of convergence of functions, we can define the extended limit of operators.

\begin{definition}
	Let $B_n \subseteq C_b(\bR^+ \times E_n) \times C_b(\bR^+ \times E_n)$. Define $ex-\LIM B_n \subseteq C_b(E) \times C_b(E)$ as the set
	\begin{multline*}
		ex-\LIM B_n \\
		= \left\{(f,g) \in C_b(E) \times C_b(E) \, \middle| \, \exists \, (f_n,g_n) \in B_n: \, \LIM f_n = f, \LIM g_n = g \right\}.
	\end{multline*}
\end{definition}

\subsection{The main reduction step} \label{section:reduction_LDP_to_HJ_intuition}

As announced at the beginning of the section, the large deviation principle can be derived from three main inputs:
\begin{enumerate}
	\item exponential tightness,
	\item convergence of operators,
	\item well-posedness of the Hamilton-Jacobi equation.
\end{enumerate}

This framework was first established in \cite[Theorem 7.18]{FK06}, but we will work in the context of the reworked version of \cite[Theorem 7.10]{Kr19c}. The key insight is the following reduction:
\begin{itemize}
	\item Exponential tightness allows one to reduce the large deviation principle on the Skorokhod space to that of the large deviations of the finite dimensional distributions.
	\item Brycs theorem, in combination with the Markov property, allows one to reduce the large deviation principle for the finite dimensional convergence to the large deviations at time $0$ and the convergence of conditional generating functions \eqref{eqn:conditional_generating_function}. 
	\item The conditional generating function forms a semigroup. The convergence of semigroups can be treated via a functional analytic framework.
\end{itemize}

Before giving the large deviation result, we introduce the involved operators for the functional analytic framework, and a weakened version of exponential tightness.

Thus, we start by introducing a semigroup, a resolvent and an infinitesimal generator. Denote by
\begin{equation} \label{eqn:conditional_generating_function}
	V_n(t)f(s,\mu,w) := \frac{1}{n} \log \bE\left[e^{n f(Y_n(t))} \, \middle| \, Y_n(0) = (s,w,\mu) \right]
\end{equation}
the conditional log-generating function. As a function of $t$ the operators $V_n(t)$ form a non-linear semigroup. The formal infinitesimal generator of the semigroup is given by the operator $H_n$ 
\begin{equation} \label{eqn:defH_abstract}
	\begin{aligned}
		\cD(H_n) & := \left\{f \in C_b(\bR^+ \times E_n) \, \middle| \, e^{nf} \in \cD(\vec{A}_n) \right\}, \\
		H_n f & := \frac{1}{n} e^{-n f} \vec{A}_n e^{n f},
	\end{aligned}
\end{equation}
and the corresponding resolvent $R_n(\lambda) = (\bONE - \lambda H_n)^{-1}$ is given by 
\begin{multline*}
	R_n(\lambda)h(s,\mu,w) \\
	:= \sup_{\bQ \in \cP(D_{\bR^+ \times E_n}(\bR^+))} \left\{ \int_0^\infty \lambda^{-1} e^{-\lambda^{-1}t} \left( \int h(Y_n(t)) \bQ(\dd Y_n) - \frac{1}{n} S_t(\bQ \, | \, \PR_{n,s,\mu,w} ) \right) \dd t\right\} 
\end{multline*}
where $S_t$ is the relative entropy on the set $D_{\bR^+ \times E_n}([0,T])$ and where $\PR_{n,s,\mu,w}$ is the law of $t \mapsto Y_n(t)$ when started in $(s,\mu,w)$. A full analysis of this triplet is given in \cite{Kr19c}.

Next, we introduce the exponential compact containment condition. This condition is a weakened version of exponential tightness on the path space. The weakening consists of only requiring exponential tightness for the time marginals. In our context, the empirical measures live on a compact space, so the condition translates to a control on the number of jumps that the process makes.

\begin{definition}
	We say that the processes $t \mapsto Y_n(t) = (t_0 + t,\mu_n(t),w_n(t))$ satisfy the \textit{exponential compact containment condition} if for each $T_0 > 0$, $M_0 > 0$, $T > 0$ and $a > 0$, there is a $M = M(T_0,M_0,T,a)$ such that
	\begin{equation*}
		\limsup_{n \rightarrow \infty} \sup_{(t_0,\mu,w) \in [0,T_0] \times E_n: |w| \leq M} \frac{1}{n} \log P_{t_0,\mu,w}\left[|w_n(t)| > M \text{ for some } t \in [0,T] \right] \leq - a.
	\end{equation*}
\end{definition}

\begin{theorem}[Adaptation of Theorem 7.10 of \cite{Kr19c} to our context] \label{theorem:abstract_LDP}
	Suppose that that the exponential compact containment condition holds.
	
	Denote $Z_n = (\mu_n(t),w_n(t))$. Suppose that
	\begin{enumerate}[(a)]
		\item \label{item:LDP_abstract_initialLDP} The large deviation principle holds for $Z_n(0)$ on $E$ with speed $n$ and good rate function $I_0$.
		\item \label{item:LDP_abstract_exptight} The processes $Z_n(t) = (\mu_n(t),w_n(t))$ are exponentially tight on $D_E(\bR^+)$ with speed $n$.
		\item \label{item:LDP_abstract_convergenceHn} There is an operator $H \subseteq C_b(E) \times C_b(E)$ such that $H \subseteq ex-\LIM H_n$.
		\item \label{item:LDP_abstract_comparison} For all $h \in C_b(E)$ and $\lambda > 0$ the comparison principle holds for $f - \lambda Hf = h$.
	\end{enumerate}
	
	Then there is a family of operators $R(\lambda)$, $\lambda > 0$ and a semigroup $V(t)$, $ t \geq 0$ on $C_b(E)$ such that for all $f \in C_b(E)$, $x \in E$ and $t > 0$, we have
	\begin{equation} \label{eqn:convergence_semigroup}
		V(t)f(x) = \lim_{m \rightarrow \infty} R^m\left(\frac{t}{m}\right) f(x).
	\end{equation}
	$V(t)$ and $R(\lambda)$ satisfy
	\begin{itemize}
		\item If $\lambda > 0$ and $\LIM h_n = h$, then $\LIM R_n(\lambda) h_n = R(\lambda) h$;
		\item For $h \in C_b(E)$, the function $R(\lambda)h$ is the unique function that is a viscosity solution to $f - \lambda H f = h$;
		\item If $\LIM f_n = f$ and $t_n \rightarrow t$ we have $\LIM V_n(t_n)f_n = V(t)f$
	\end{itemize}
	In addition, the processes $Z_n =(\mu_n,w_n)$ satisfy a large deviation principle on $D_E(\bR^+)$ with speed $n$ and rate function
	\begin{equation} \label{eqn:LDP_rate2}
		I(\gamma) = I_0(\gamma(0)) + \sup_{k \geq 1} \sup_{\substack{0 = t_0 < t_1 < \dots, t_k \\ t_i \in \Delta_\gamma^c}} \sum_{i=1}^{k} I_{t_i - t_{i-1}}(\gamma(t_i) \, | \, \gamma(t_{i-1})).
	\end{equation}
	Here $\Delta_\gamma^c$ is the set of continuity points of $\gamma$. The conditional rate functions $I_t$ are given by
	\begin{equation*}
		I_t(y \, | \, x) = \sup_{f \in C_b(E)} \left\{f(y) - V(t)f(x) \right\}.
	\end{equation*}
\end{theorem}

On the basis of this abstract result, we derive our main result.

\begin{proof}[Proof of Theorem \ref{theorem:ldp_mean_field_jump_process_periodic}]
	To apply Theorem \ref{theorem:abstract_LDP}, we have to verify \ref{item:LDP_abstract_initialLDP} to \ref{item:LDP_abstract_comparison}. Assumption \ref{item:LDP_abstract_initialLDP} follows by the assumption on the large deviation principle at time $0$ of Theorem \ref{theorem:ldp_mean_field_jump_process_periodic}. The other three properties will be checked below. We give their respective statements.
	
	\smallskip
	
	We verify \ref{item:LDP_abstract_exptight} in Proposition \ref{proposition:exponential_tightness} and we verify \ref{item:LDP_abstract_convergenceHn} in Proposition \ref{proposition:convergence_of_operators} below.
	
	\ref{item:LDP_abstract_comparison} is the result of Theorem \ref{theorem:comparison_principle} which will be proven in Section \ref{section:proofs_comparison}.
	
	\smallskip
	
	At this point, we have the large deviation principle with a rate function in projective limit form. That this rate function equals the rate-function in Lagrangian form follows from Proposition \ref{proposition:variational_representation}.
\end{proof}

\section{Verification of the comparison principle} \label{section:proofs_comparison}

\subsection{A general method to verify the comparison principle} \label{section:general_method_for_comparison_principle} 

Throughout this section, we assume that $E = \cP(\{1,\dots,q\}) \times (\bR^+)^{|\Gamma|}$ is parametrized as a subset of $\bR^{q-1 + |\Gamma|}$ to make sure that it is contained in the closure of its interior. We correspondingly write $d = q-1 + |\Gamma|$. 

\smallskip

In this section, we give the main technical results that can be used to verify the comparison principle. These methods are based on those used in \cite{CIL92,FK06,DFL11,Kr16b,CoKr17}. The key component in this method is the method of `doubling variables'. To obtain the comparison principle, one aims to give an upper bound on
\begin{equation} \label{eqn:upper_bound_basic_comparison}
	\sup_x u(x) - v(x).
\end{equation}
A direct estimate is hard to obtain, so instead one doubles the amount of variables and a penalization to large discrepancies between the variables. Thus one tries to find an upper bound for
\begin{equation*}
	\sup_{x,y} u(x) - v(y) - \alpha \Psi(x,y)
\end{equation*}
which converges to $0$ as $\alpha \rightarrow \infty$. If $\Psi(x,y) = 0$ if and only if $x = y$, one obtains as a consequence an upper bound on \eqref{eqn:upper_bound_basic_comparison} .

This technique works in the setting of compact spaces. In the context of non-compact spaces, one also has to penalize $x,y$ that are `far away'. Thus later works introduce the use of `containment' or Lyapunov functions. We introduce both these concepts below.

In this section, a novel aspect in comparison to the aforementioned papers, is the use of two `penalization' functions instead of one.

\begin{definition}
	We say that $\{\Psi_1,\Psi_2\}$, with $\Psi_i : E^2 \rightarrow \bR^+$ is a \textit{good pair of penalization functions} if $\Psi_i \in C^1(E^2)$ and if $x = y$ if and only if $\Psi_i(x,y) = 0$ for $i \in \{1,2\}$.
\end{definition}

In the proof of Theorem \ref{theorem:comparison_principle}, we work with a penalization
\begin{equation*}
	\alpha_1 \Psi_1 + \alpha_2 \Psi_2,
\end{equation*}
then send $\alpha_1 \rightarrow \infty$ and afterwards $\alpha_2 \rightarrow \infty$. To be able to treat both steps in a similar fashion, we introduce an auxiliary penalization function in which $\alpha_1$ is already sent to infinity:
\begin{equation*}
	\widehat{\Psi}_2(x,y)  
	:= \begin{cases}
		\Psi_2(x,y) & \text{if } \Psi_1(x,y) = 0, \\
		\infty & \text{if }  \Psi_1(x,y) \neq 0.
	\end{cases}
\end{equation*}

Finally, we introduce containment functions that allow us to restrict our analysis to compact sets.

\begin{definition} \label{definition:good_containment_function}
	Let $\cH : E \times \bR^d \rightarrow \bR$, we say that $\Upsilon : E \rightarrow \bR$ is a \textit{good containment function} (for $\cH$) if
	\begin{enumerate}[($\Upsilon$a)]
		\item $\Upsilon \geq 0$ and there exists a point $z_0$ such that $\Upsilon(z_0) = 0$,
		\item $\Upsilon$ is continuously differentiable, 
		\item for every $c \geq 0$, the sublevel set $\{z \in E \, | \, \Upsilon(z) \leq c\}$ is compact,
		\item we have $ \sup_{z \in E} \cH(z,\nabla \Upsilon(z)) < \infty$.
	\end{enumerate}
\end{definition}

The following result gives us the main technical input for the proof of the comparison principle.

\begin{lemma} \label{lemma:consecutive_limit_points}
	Let $u : E \rightarrow \bR$ be bounded and upper semi-continuous, let $v : E \rightarrow \bR$ be bounded and lower semi-continuous. Let $\{\Psi_1,\Psi_2\}$ be a good pair of penalization functions and let $\Upsilon : E \rightarrow \bR^+$ be a good containment function for $\cH$.
	
	\smallskip
	
	Fix $\varepsilon > 0$. Then there is a compact set $K_\varepsilon \subseteq E$ such that for every $\alpha \in (0,\infty)^2$, $\alpha = (\alpha_1,\alpha_2)$ there exist points $x_{\alpha,\varepsilon},y_{\alpha,\varepsilon} \in K_\varepsilon$, such that
	\begin{multline*}
		\frac{u(x_{\alpha,\varepsilon})}{1-\varepsilon} - \frac{v(y_{\alpha,\varepsilon})}{1+\varepsilon} - \sum_{i=1}^2  \alpha_i\Psi_i(x_{\alpha,\varepsilon},y_{\alpha,\varepsilon}) - \frac{\varepsilon}{1-\varepsilon}\Upsilon(x_{\alpha,\varepsilon}) -\frac{\varepsilon}{1+\varepsilon}\Upsilon(y_{\alpha,\varepsilon}) \\
		= \sup_{x,y \in E} \left\{\frac{u(x)}{1-\varepsilon} - \frac{v(y)}{1+\varepsilon} -  \sum_{i = 1}^2 \alpha_i \Psi_i(x,y)  - \frac{\varepsilon}{1-\varepsilon}\Upsilon(x) - \frac{\varepsilon}{1+\varepsilon}\Upsilon(y)\right\}.
	\end{multline*}
	In addition, for all $\varepsilon > 0$ and $\alpha_2 > 0$ there are limit points $x_{\alpha_{2},\varepsilon}, y_{\alpha_2,\varepsilon} \in K_\varepsilon$ of $x_{(\alpha_1,\alpha_2),\varepsilon}$ and $y_{(\alpha_1,\alpha_2),\varepsilon}$ as $\alpha_{1}\rightarrow \infty$ and we have 
	\begin{gather*}
		\lim_{\alpha_1 \rightarrow \infty}\alpha_1 \Psi_1(x_{(\alpha_{1},\alpha_2),\varepsilon},y_{(\alpha_{1},\alpha_2),\varepsilon}) = 0, \\
		\Psi_1(x_{\alpha_2,\varepsilon},y_{\alpha_2,\varepsilon}) = 0. 
	\end{gather*}
	In addition
	\begin{multline*}
		\frac{u(x_{\alpha_2,\varepsilon})}{1-\varepsilon} - \frac{v(y_{\alpha_2,\varepsilon})}{1+\varepsilon} - \alpha_2\widehat{\Psi}_2 (x_{\alpha_2,\varepsilon},y_{\alpha_2,\varepsilon}) 	- \frac{\varepsilon}{1-\varepsilon}\Upsilon(x_{\alpha_2,\varepsilon}) -\frac{\varepsilon}{1+\varepsilon}\Upsilon(y_{\alpha_2,\varepsilon}) \\
		= \sup_{x,y \in E} \left\{\frac{u(x)}{1-\varepsilon} - \frac{v(y)}{1+\varepsilon} -  \alpha_2\widehat{\Psi}_2(x,y)  - \frac{\varepsilon}{1-\varepsilon}\Upsilon(x) - \frac{\varepsilon}{1+\varepsilon}\Upsilon(y)\right\}.
	\end{multline*}
\end{lemma}

\begin{remark}
	The result remains true for $\Psi_i$ and $\Upsilon$ that are lower semi-continuous instead of $C^1$.
\end{remark}

The proof of Lemma \ref{lemma:consecutive_limit_points}, carried out below, is based on the following standard result.

\begin{lemma}[Proposition 3.7 of \cite{CIL92} or Lemma 9.2 in \cite{FK06}] \label{lemma:abstract_doubling_lemma}
	Let $F : E \rightarrow \bR \cup \{-\infty\}$ be bounded above, upper semi-continuous, and such that for each $c \in \bR$ the set $\left\{(x,y) \in E^2 \, \middle| \, F(x,y) \geq c \right\}$ is compact. Let $G : E^2 \rightarrow [0,\infty]$ be lower semi-continuous and such that $x = y$ implies $G(x,y) = 0$.
	
	Then there exist for each $\alpha > 0$ variables $(x_\alpha,y_\alpha) \in E^2$ such that
	\begin{equation*}
		F(x_\alpha,y_\alpha) - \alpha G(x_\alpha,y_\alpha) = \sup_{x,y \in E} \left\{F(x,y) - \alpha G(x,y)  \right\}.
	\end{equation*}
	In addition, we have
	\begin{enumerate}[(a)]
		\item The set $\{x_{\alpha}, y_{\alpha} \, | \,  \alpha  > 0\}$ is relatively compact in $E$;
		\item Any limit point $(x_0,y_0)$ of  $(x_\alpha,y_\alpha)_{\alpha > 0}$ as $\alpha \rightarrow \infty$ satisfies $G(x_0,y_0) = 0$ and $F(x_0,y_0) = \sup_{x,y \in E, G(x,y) = 0} F(x,y)$;
		\item We have $\lim_{\alpha \rightarrow \infty} \alpha G(x_\alpha,y_\alpha) = 0$.
	\end{enumerate}
\end{lemma}

\begin{proof}[Proof of Lemma \ref{lemma:consecutive_limit_points}]
	As $u,v$ are bounded and the $\Psi_i$ are bounded from below, we find using the semi-continuity properties of all functions involved, and the compact level sets of $\Upsilon$, that there is a compact set $K_\varepsilon \subseteq E$ and variables $x_{\alpha,\varepsilon},y_{\alpha,\varepsilon} \in K_\varepsilon$ as in the first claim of the lemma. 
	
	The second statement follows from Lemma \ref{lemma:abstract_doubling_lemma} by taking for $\alpha$ the variable $\alpha_1$, and for $F$ and $G$ the functions
	\begin{align*}
		F(x,y) &: = \left\{\frac{u(x)}{1-\varepsilon} - \frac{v(y)}{1+\varepsilon} -  \alpha_2\Psi_2(x,y)  - \frac{\varepsilon}{1-\varepsilon}\Upsilon(x) - \frac{\varepsilon}{1+\varepsilon}\Upsilon(y)\right\}, \\
		G(x,y) & := \Psi_1(x,y).
	\end{align*}
\end{proof}

The following result gives us the explicit condition that can be used to verify the comparison principle.

\begin{proposition} \label{proposition:comparison_conditions_on_H}
	Let $H : \cD(H) \subseteq C_b(E) \rightarrow C_b(E)$ have domain $\cD(H)$ satisfying $C_c^\infty(E) \subseteq \cD(H) \subseteq C_b^1(E)$ and be of the form $Hf(x) = \cH(x,\nabla f(x))$. Assume that the map $\cH : E \times \bR^d \rightarrow \bR$ is continuous and that for each $x \in E$ the map $p \mapsto \cH(x,p)$ is convex.
	
	Fix $\lambda >0$, $h \in C_b(E)$ and consider $u$ and $v$ sub- and supersolution to $f - \lambda Hf = h$.
	
	\smallskip
	
	Let $k \in \bN \setminus \{0\}$ and let $\{\Psi_1,\Psi_2\}$ be a pair of good penalization functions and $\Upsilon$ be a good containment function. Moreover, for every $\alpha = (\alpha_1,\alpha_2) \in (0,\infty)^2$ and $\varepsilon >0$ let $x_{\alpha,\varepsilon},y_{\alpha,\varepsilon} \in E$ be such that
	\begin{multline} \label{eqn:comparison_principle_proof_choice_of_sequences}
		\frac{u(x_{\alpha,\varepsilon})}{1-\varepsilon} - \frac{v(y_{\alpha,\varepsilon})}{1+\varepsilon} -  \sum_{i=1}^2 \alpha_i \Psi_i(x_{\alpha,\varepsilon},y_{\alpha,\varepsilon}) - \frac{\varepsilon}{1-\varepsilon}\Upsilon(x_{\alpha,\varepsilon}) -\frac{\varepsilon}{1+\varepsilon}\Upsilon(y_{\alpha,\varepsilon}) \\
		= \sup_{x,y \in E} \left\{\frac{u(x)}{1-\varepsilon} - \frac{v(y)}{1+\varepsilon} - \sum_{i=1}^2 \alpha_i \Psi_i(x,y)  - \frac{\varepsilon}{1-\varepsilon}\Upsilon(x) - \frac{\varepsilon}{1+\varepsilon}\Upsilon(y)\right\}.
	\end{multline}
	
	Suppose that
	\begin{multline}\label{condH:negative:liminf}
		\liminf_{\varepsilon \rightarrow 0} \liminf_{\alpha_2 \rightarrow \infty}  \liminf_{\alpha_1 \rightarrow \infty} \cH\left(x_{\alpha,\varepsilon},\sum_{i=1}^2 \alpha_i \nabla \Psi_i(\cdot,y_{\alpha,\varepsilon})(x_{\alpha,\varepsilon})\right) \\
		- \cH\left(y_{\alpha,\varepsilon},- \sum_{i=1}^2 \alpha_i \nabla \Psi_i(x_{\alpha,\varepsilon},\cdot)(y_{\alpha,\varepsilon})\right) \leq 0,
	\end{multline}
	then $u \leq v$. In other words: $f - \lambda H f = h$ satisfies the comparison principle. 
\end{proposition}

\begin{proof} 
	Using the convexity of $H$ and the definitions of sub- and supersolutions, we find by Lemma \ref{lemma:establish_first_bound_viscosity_argument} that
	\begin{align} 
		& \sup_x u(x) - v(x) \nonumber\\
		& \leq \frac{h(x_{\alpha,\varepsilon})}{1 - \varepsilon} - \frac{h(y_{\alpha,\varepsilon})}{1+\varepsilon}  \label{eqn:eqn:comp_proof_final_bound:line1}\\
		& \qquad + \frac{\varepsilon}{1-\varepsilon}\cH(x_{\alpha,\varepsilon}, \nabla \Upsilon(x_{\alpha,\varepsilon})) + \frac{\varepsilon}{1+\varepsilon}\cH(y_{\alpha,\varepsilon}, \nabla\Upsilon(y_{\alpha,\varepsilon})) \label{eqn:eqn:comp_proof_final_bound:line2}\\
		& \qquad +  \lambda \left[\cH\left(x_{\alpha,\varepsilon}, \sum_{i=1}^2 \alpha_i\nabla\Psi_i(\cdot,y_{\alpha,\varepsilon})(x_{\alpha,\varepsilon})\right) - \cH\left(y_{\alpha,\varepsilon},- \sum_{i=1}^2 \alpha_i\nabla\Psi_i(x_{\alpha,\varepsilon},\cdot)(y_{\alpha,\varepsilon})\right)\right].  \label{eqn:eqn:comp_proof_final_bound:line3}
	\end{align}
	Consecutively taking $\liminf$ over $\alpha_1,\alpha_2,\varepsilon$, the term \eqref{eqn:eqn:comp_proof_final_bound:line3} vanishes by assumption. The term in \eqref{eqn:eqn:comp_proof_final_bound:line2} vanishes as well, due to the uniform bounds on $\cH(z,\nabla \Upsilon(z))$ by property ($\Upsilon$d) of Definition \ref{definition:good_containment_function}. Consecutively taking limit points as in Lemma \ref{lemma:consecutive_limit_points} by sending $\alpha_1$, then $\alpha_2$ to infinity, we find a pair $(x_\varepsilon,y_\varepsilon)$ with $\Psi_1(x_\varepsilon,y_\varepsilon) = \Psi_2(x_\varepsilon,y_\varepsilon) = 0$. This implies $x_\varepsilon = y_\varepsilon$. Thus, taking $\liminf$ over $\alpha_1$ and $\alpha_2$ the term in \eqref{eqn:eqn:comp_proof_final_bound:line1} is bounded above by 
	\begin{equation*}
		\sup_{z \in K_\varepsilon} \frac{h(z)}{1-\varepsilon} - \frac{h(z)}{1+\varepsilon} \leq \frac{2 \varepsilon}{1 - \varepsilon^2} \vn{h},
	\end{equation*}
	which vanishes if $\varepsilon \rightarrow 0$. We conclude that the comparison principle holds for $f - \lambda H f = h$.
\end{proof}

The next lemma establishes that the Hamiltonian applied to the pair of penalization functions is either bounded below or above. Using coercivity properties of $\cH$, this allows us to derive properties of the sequences $(x_{\alpha,\varepsilon},y_{\alpha,\varepsilon})$ that can be used afterwards to help the the  verification of \eqref{condH:negative:liminf}.

\begin{lemma} \label{lemma:control_on_H}
	Let $H : \cD(H) \subseteq C_b(E) \rightarrow C_b(E)$ with domain $\cD(H)$ satisfying $C_c^\infty(E) \subseteq \cD(H) \subseteq C_b^1(E)$ and such that $Hf(x) = \cH(x,\nabla f(x))$. Assume that the map $\cH : E \times \bR^d \rightarrow \bR$ is continuous and that for each $x \in E$ the map $p \mapsto \cH(x,p)$ is convex.
	
	Let $h \in C_b(E)$ and $\lambda > 0$ and let $u$ be a subsolution and $v$ a supersolution to $f - \lambda H = h$. 
	
	Let $\{\Psi_1,\Psi_2\}$ be a pair of good penalization functions and $\Upsilon$ be a good containment function. Moreover, for every $\alpha = (\alpha_1,\alpha_2) \in (0,\infty)^2$ and $\varepsilon >0$ let $x_{\alpha,\varepsilon},y_{\alpha,\varepsilon} \in E$ be as in \eqref{eqn:comparison_principle_proof_choice_of_sequences}. Then we have that
	\begin{align} 
		& \sup_{\varepsilon, \alpha} \cH\left(y_{\alpha,\varepsilon}, - \sum_{i=1}^2\alpha_i (\nabla \Psi_i(x_{\alpha,\varepsilon},\cdot))(y_{\alpha,\varepsilon})\right) < \infty, \label{eqn:control_on_H_sup} \\
		& \inf_{\varepsilon, \alpha} \cH\left(x_{\alpha,\varepsilon},  \sum_{i=1}^2\alpha_i (\nabla \Psi_i(\cdot,y_{\alpha,\varepsilon}))(x_{\alpha,\varepsilon})\right) > -\infty. \label{eqn:control_on_H_inf} 
	\end{align}
\end{lemma}

The proof is an adaptation of Lemma 9.3 in \cite{FK06}. For a similar version in compact setting see Lemma 5 in \cite{Kr16b}.

\begin{proof} 
	We only prove the first statement, the second can be proven similarly. Using that $v$ is a supersolution to $f - \lambda H f = h$, we find that it is a supersolution to the equation $f - \lambda H_\ddagger f = h$, where $H_\ddagger$ is a super-extension of $H$ that includes functions of the type $y \mapsto (-(1+\varepsilon) \sum_{i=1}^2\alpha_i \Psi_i(x,y)- \varepsilon \Upsilon(y)$ in its domain, see Lemma \ref{lemma:viscosity_extension}. It follows that for the points $(x_{\alpha,\varepsilon},y_{\alpha,\varepsilon})$, we have
	\begin{multline*}
		\cH\left(y_{\alpha,\varepsilon}, - (1+\varepsilon) \sum_{i=1}^2 \alpha_i \nabla \Psi_i(x_{\alpha,\varepsilon},\cdot)(y_{\alpha,\varepsilon}) - \varepsilon \nabla \Upsilon(y_{\alpha,\varepsilon})\right) \\
		\leq \frac{v(y_{\alpha,\varepsilon}) - h(y_{\alpha,\varepsilon})}{\lambda} \leq  \frac{\vn{v-h}}{\lambda}.
	\end{multline*}
	By the convexity of $p \mapsto \cH(x,p)$, we find
	\begin{multline*}
		\cH\left(y_{\alpha,\varepsilon}, -\sum_{i=1}^2\alpha_i \nabla \Psi_i(x_{\alpha,\varepsilon},\cdot)(y_{\alpha,\varepsilon})\right) \\
		\leq \frac{1}{1+\varepsilon} \cH\left(y_{\alpha,\varepsilon},-(1+\varepsilon)\sum_{i=1}^2\alpha_i \nabla \Psi_i(x_{\alpha,\varepsilon},\cdot)(y_{\alpha,\varepsilon}) - \varepsilon \nabla \Upsilon(y_{\alpha,\varepsilon})\right) \\
		+ \frac{\varepsilon}{1+\varepsilon}  \cH\left(y_{\alpha,\varepsilon}, \nabla \Upsilon(y_{\alpha,\varepsilon})\right),
	\end{multline*}
	which implies
	\begin{multline*}
		\sup_{\alpha} \cH\left(y_{\alpha,\varepsilon}, - \sum_{i=1}^2 \alpha_i (\nabla \Psi_i(x_{\alpha,\varepsilon},\cdot))(y_{\alpha,\varepsilon})\right) \\
		\leq \frac{1}{1+\varepsilon}\left( \frac{\vn{v-h}}{\lambda} + \varepsilon \sup_{z} \cH(z,\nabla \Upsilon(z))\right) < \infty.
	\end{multline*}
	Taking the supremum over $\varepsilon$ yields the final claim.
\end{proof}

\subsection{The verification of the comparison principle for our explicit Hamiltonian} \label{section:comparison_explicit}

We now turn to the verification of Theorem \ref{theorem:comparison_principle}, that is, the verification of the comparison principle for the Hamilton-Jacobi equation with Hamiltonians of the type
\begin{equation*}
	\cH((\mu,w),p) = \sum_{(a,b) \in \Gamma} v(a,b,\mu)\left[\exp\left\{p_b - p _a + p_{(a,b)} \right\} - 1 \right].
\end{equation*}
To obtain the comparison principle, we aim to apply Proposition \ref{proposition:comparison_conditions_on_H}. To do so, we first need to choose a pair of good penalization functions $(\Psi_1,\Psi_2)$ and a good containment function $\Upsilon$. This is the first thing we will do in the section. Afterwards, we verify \eqref{condH:negative:liminf} which is the key hypothesis of Proposition \ref{proposition:comparison_conditions_on_H}. 

\smallskip

We start by construction of a good containment function.

\begin{lemma} \label{lemma:explicit_good_containment_function}
	Consider $\cH :  E \times  \bR^d \rightarrow \bR$ given by
	\begin{equation*}
		\cH((\mu,w),p) =  \sum_{(a,b) \in \Gamma} v(a,b,\mu)\left[\exp\left\{p_b - p _a + p_{(a,b)} \right\} - 1 \right]
	\end{equation*}
	where $v : \Gamma \times \cP(\{1,\dots,q\})$ is continuous and non-negative.
	
	Then the function $\Upsilon(\mu,w) = \sum_{(a,b) \in \Gamma} \log \left( 1 + w_{(a,b)}\right)$ is a good containment function for $\cH$.
\end{lemma}

\begin{proof}
	As $\cP(\{1,\dots,q\})$ is compact and $x \mapsto \log(1+x)$ has compact level sets on $\bR^+$ the map $\Upsilon$ has compact level sets in $E$ also. Clearly $\Upsilon$ is smooth. Thus, it suffices to show $\sup_{\mu,w} \cH((\mu,w),\nabla \Upsilon(\mu,w)) < \infty$:
	\begin{align*}
		\cH((\mu,w), \nabla \Upsilon(\mu,w)) & =  \sum_{(a,b) \in \Gamma} v(a,b,\mu)\left[\exp\left\{(1 + w_{(a,b)})^{-1} \right\} - 1 \right] \\
		& \leq  \sum_{(a,b) \in \Gamma} v(a,b,\mu)\left[\exp\left\{1 \right\} - 1 \right].
	\end{align*}
	The claim follows as $v$ is continuous, and, therefore, bounded.
\end{proof}

We proceed by constructing a pair of good penalization functions. For $\Psi_1$ we use a version of the quadratic distance on the space of measures which was first introduced in \cite{Kr16b}. For $\Psi_2$ we use a standard quadratic distance on the space of fluxes. The exact definition follows.

For $x \in \bR$, let $x^- := x \wedge 0$ and $x^+ = x \vee 0$.

\begin{lemma} \label{lemma:explicit_good_penalization_functions}
	Define $\Psi_1,\Psi_2$ by
	\begin{align*}
		\Psi_1(\mu,\hat{\mu}) & = \frac{1}{2} \sum_{a} ((\mu(a) - \hat{\mu}(a))^-)^2 = \frac{1}{2} \sum_{a} ((\hat{\mu}(a) - \mu(a))^+)^2, \\
		\Psi_2(w,\hat{w}) & := \frac{1}{2} \sum_{(a,b) \in \Gamma} (w_{(a,b)} - \hat{w}_{(a,b)})^2.
	\end{align*}
	The pair $(\Psi_1,\Psi_2)$ is a pair of good penalization functions for $E = \cP(\{1,\dots,1\}) \times (\bR^+)^{\Gamma}$.
	
	In addition, we have
	\begin{align*}
		(\nabla \Psi_1(\cdot,\hat{\mu}))(\mu) & = - (\nabla \Psi_1(\mu,\cdot))(\hat{\mu}), \\
		(\nabla \Psi_2(\cdot,\hat{w}))(w) & = - (\nabla \Psi_2(w,\cdot))(\hat{w}).
	\end{align*}
\end{lemma}

The use of $\Psi_1$ is highly specific for the space $\cP(\{1,2,\dots,q\})$.  The special form is motivated by the linear constraint $\sum \mu(a) = 1$. The use of the standard quadratic distance leads to `loss of control' over the variables when applying Lemma \ref{lemma:control_on_H}. This issue is related to the discussion in Remark \ref{remark:choice_of_penalization}.
The adaptation of the quadratic distance takes into account the form of our Hamiltonian and the linear constraint in a symmetric way and is geared towards re-establishing the control via Lemma \ref{lemma:control_on_H}. 

\begin{proof}[Proof of Lemma \ref{lemma:explicit_good_penalization_functions}]
	Note that as $\sum_i \mu(i) = \sum_i \hat{\mu}(i) = 1$, we find that $\Psi_1(\mu,\hat{\mu}) = 0$ implies that $\mu = \hat{\mu}$. As $\Psi_2$ is a quadratic distance on $(\bR^+)^\Gamma$, we indeed have that $(\mu,w) = (\hat{\mu},\hat{w})$ if and only if $\Psi_1(\mu,\hat{\mu}) + \Psi_2(w,\hat{w}) = 0$.
	
	The second claim follows by direct verification.
\end{proof}

We proceed with the verification of of the comparison principle by establishing the key estimate \eqref{condH:negative:liminf} of Proposition \ref{proposition:comparison_conditions_on_H}.

\begin{proof}[Proof of Theorem \ref{theorem:comparison_principle}]
	The proof is an adaptation of the proof of Theorem 4 in \cite{Kr16b}. Fix $h \in C_b(E)$ and $\lambda > 0$. Let $u$ be a subsolution and $v$ be a supersolution to $f - \lambda H f = h$.
	
	We verify \eqref{condH:negative:liminf} of Proposition \ref{proposition:comparison_conditions_on_H} using containment function $\Upsilon$ and penalization functions $\Psi_1,\Psi_2$ from Lemmas \ref{lemma:explicit_good_containment_function} and \ref{lemma:explicit_good_penalization_functions}. For $\varepsilon > 0$, $\alpha_1,\alpha_2 > 0$ let  $x_{\alpha,\varepsilon} := (\mu_{\alpha,\varepsilon},w_{\alpha,\varepsilon})$ and $y_{\alpha,\varepsilon} := (\hat{\mu}_{\alpha,\varepsilon},\hat{w}_{\alpha,\varepsilon})$ be as in \eqref{eqn:comparison_principle_proof_choice_of_sequences}.
	
	\smallskip
	
	To establish the theorem, we will show that already after taking one liminf, the bound is satisfied. Indeed, we will show for fixed $\varepsilon > 0$ and $\alpha_2 > 0$ that
	\begin{multline}
		\liminf_{\alpha_1 \rightarrow \infty} \cH\left(x_{\alpha,\varepsilon},\sum_{i=1}^2 \alpha_i \nabla \Psi_i(\cdot,y_{\alpha,\varepsilon})(x_{\alpha,\varepsilon})\right) \\
		- \cH\left(y_{\alpha,\varepsilon},- \sum_{i=1}^2 \alpha_i \nabla \Psi_i(x_{\alpha,\varepsilon},\cdot)(y_{\alpha,\varepsilon})\right) \leq 0. \label{eqn:to_prove_fundamental_bound}
	\end{multline}
	By Lemma \ref{lemma:consecutive_limit_points} for fixed $\alpha_2, \varepsilon$ and sending $\alpha_1 \rightarrow \infty$, we find limit points $(x_{\alpha_2,\varepsilon},y_{\alpha_2,\varepsilon}) = ((\mu_{\alpha_2,\varepsilon}, w_{\alpha_2,\varepsilon}),(\mu_{\alpha_2,\varepsilon},\hat{w}_{\alpha_2,\varepsilon}))$ of the sequence $((\mu_{\alpha,\varepsilon},w_{\alpha,\varepsilon}),(\hat{\mu}_{\alpha,\varepsilon},\hat{w}_{\alpha,\varepsilon}))$. Without loss of generality, going to a subsequence if necessary, we assume that these sequences converge to their respective limit point. By the definition of $\cH$, we have
	\begin{align}
		& \cH\left(x_{\alpha,\varepsilon},\sum_{i=1}^2 \alpha_i \nabla \Psi_i(\cdot,y_{\alpha,\varepsilon})(x_{\alpha,\varepsilon})\right) - \cH\left(y_{\alpha,\varepsilon},- \sum_{i=1}^2 \alpha_i \nabla \Psi_i(x_{\alpha,\varepsilon},\cdot)(y_{\alpha,\varepsilon})\right) \notag \\
		& = \sum_{(a,b) \in \Gamma} \left[v(a,b,\mu_{\alpha,\varepsilon}) - v(a,b,\hat{\mu}_{\alpha,\varepsilon}) \right] \times \label{eqn:proof_comparison_Potts_sum} \\
		& \hspace{4em} \left[e^{\alpha_1\left(\left(\mu_{\alpha,\varepsilon}(b) - \hat{\mu}_{\alpha,\varepsilon}(b)\right)^- - \left(\mu_{\alpha,\varepsilon}(a) - \hat{\mu}_{\alpha,\varepsilon}(a)\right)^-\right) + \alpha_2\left(w_{\alpha,\varepsilon}(a,b) - \hat{w}_{\alpha,\varepsilon}(a,b)\right)} - 1\right]. \notag 
	\end{align}
	To ease the notation, and focus on the parts that matter, we will write $c_{\alpha,\varepsilon}(a,b) := \alpha_2\left(w_{\alpha,\varepsilon}(a,b) - \hat{w}_{\alpha,\varepsilon}(a,b)\right)$ as this term will not play a role in our bounds below. In fact, for fixed $\varepsilon$ and $\alpha_2$, we have for all $(a,b) \in \Gamma$ that
	\begin{equation} \label{eqn:bound_on_errors_c}
		\sup_{\alpha_1} \left| c_{\alpha,\varepsilon}(a,b) \right| < \infty
	\end{equation}
	because by the construction of Lemma \ref{lemma:consecutive_limit_points} we have
	\begin{equation*}
		\sup_{\alpha_1} \alpha_2 \Psi_2(w_{\alpha,\varepsilon},\hat{w}_{\alpha,\varepsilon}) < \infty.
	\end{equation*} 
	We will show that each term in \eqref{eqn:proof_comparison_Potts_sum} separately is bounded from above by $0$ asymptotically. Pick some ordering of the ordered pairs $(a,b) \in \Gamma$, and assume that we have some sequence $\alpha_1$ such that the $\liminf_{\alpha_1 \rightarrow \infty}$ of the first $l$ terms in equation \eqref{eqn:proof_comparison_Potts_sum} are bounded above by $0$. We construct a subsequence so that also term $l+1$ is asymptotically bounded above by $0$. The result then follows by induction.
	Thus, suppose that $(i,j)$ is the pair corresponding to the $l+1$-th term of the sum in \eqref{eqn:proof_comparison_Potts_sum}.

	We go through the two options of $v$ being a proper kernel. Clearly, if $v(i,j,\pi) = 0$ for all $\pi$ then we are done. Therefore, we assume that $v(i,j,\pi) \neq 0$ for all $\pi$ such that $\pi(i) > 0$ and that conditions (a) and (b) of having a proper kernel are satisfied.
	
	\textit{Case 1:} If $\mu_{\alpha_2,\varepsilon}(i) > 0$, we know by \eqref{eqn:control_on_H_sup}, using that $v(i,j,\cdot)$ is bounded away from $0$ on a neighbourhood of $\mu_{\alpha_2,\varepsilon}$ (condition (a) of having a proper kernel), that 
	\begin{equation*}
		\sup_{\alpha_1} e^{\alpha_1\left(\left(\mu_{\alpha,\varepsilon}(j) - \hat{\mu}_{\alpha,\varepsilon}(j)\right)^- - \left(\mu_{\alpha,\varepsilon}(i) - \hat{\mu}_{\alpha,\varepsilon}(i)\right)^-\right) + c_{\alpha,\varepsilon}(i,j) } - 1 < \infty.
	\end{equation*}
	As we also have that the exponential is bounded from below by $0$, we can pick a subsequence $\alpha(n) = (\alpha_{1}(n),\alpha_2)$ and some constant $c$ such that 
	\begin{equation*}
		e^{\alpha_1(n) \left(\left(\mu_{\alpha(n),\varepsilon}(j) - \hat{\mu}_{\alpha(n),\varepsilon}(j)\right)^- - \left(\mu_{\alpha(n),\varepsilon}(i) - \hat{\mu}_{\alpha(n),\varepsilon}(i)\right)^-\right) + c_{\alpha(n),\varepsilon}(i,j) } - 1  \rightarrow c.
	\end{equation*}
	Using that $\pi \rightarrow v(i,j,\pi)$ is uniformly continuous on compact sets, we see 
	\begin{align*}
		& \liminf_{\alpha_1 \rightarrow \infty} \left[v(i,j,\mu_{\alpha,\varepsilon}) - v(i,j,\hat{\mu}_{\alpha,\varepsilon}) \right] \times \\
		& \hspace{3.5em} \left[e^{\alpha_1\left(\left(\mu_{\alpha,\varepsilon}(j) - \nu_{\alpha,\varepsilon}(j)\right)^- - \left(\mu_{\alpha,\varepsilon}(i) - \nu_{\alpha,\varepsilon}(i)\right)^-\right) +c_{\alpha,\varepsilon}(i,j)} - 1\right] \\
		& \quad \leq \lim_{n \rightarrow \infty} \left[v(i,j,\mu_{\alpha(n),\varepsilon}) - v(i,j,\hat{\mu}_{\alpha(n),\varepsilon}) \right] \times \\
		& \hspace{3.5em} \lim_{n \rightarrow \infty} \left[e^{\alpha_1(n)\left(\left(\mu_{\alpha(n),\varepsilon}(j) - \hat{\mu}_{\alpha(n),\varepsilon}(j)\right)^- - \left(\mu_{\alpha(n),\varepsilon}(i) - \hat{\mu}_{\alpha(n),\varepsilon}(i)\right)^-\right) + c_{\alpha,\varepsilon}(i,j)} - 1\right] \\
		& \quad = c \lim_{n \rightarrow \infty} \left[v(i,j,\mu_{\alpha(n),\varepsilon}) - v(i,j,\hat{\mu}_{\alpha(n),\varepsilon}) \right] = 0.
	\end{align*}
	
	\textit{Case 2:} Suppose that $\mu_{\alpha,\varepsilon}(i),\hat{\mu}_{\alpha,\varepsilon}(i) \rightarrow 0$ as $\alpha_1 \rightarrow \infty$. Again by \eqref{eqn:control_on_H_sup}, we get
	\begin{equation} \label{eqn:proof_comparison_jump_sup_bound_on_exp}
		M:= \sup_{\alpha_1} v(i,j,\hat{\mu}_{\alpha,\varepsilon}) \left[e^{\alpha_1\left(\left(\mu_{\alpha,\varepsilon}(j) - \hat{\mu}_{\alpha,\varepsilon}(j)\right)^- - \left(\mu_{\alpha,\varepsilon}(i) - \hat{\mu}_{\alpha,\varepsilon}(i)\right)^-\right) + c_{\alpha,\varepsilon}(i,j)} - 1\right] < \infty.
	\end{equation}
	First of all, if $\sup_{\alpha_1} \alpha_1\left(\left(\mu_{\alpha,\varepsilon}(j) - \hat{\mu}_{\alpha,\varepsilon}(j)\right)^- - \left(\mu_{\alpha,\varepsilon}(i) - \hat{\mu}_{\alpha,\varepsilon}(i)\right)^-\right) + c_{\alpha,\varepsilon}(i,j) < \infty$, then the argument given in step 1 above also takes care of this situation. So suppose that this supremum is infinite. Clearly, the contribution $\alpha_1\left(\mu_{\alpha,\varepsilon}(j) - \hat{\mu}_{\alpha,\varepsilon}(j)\right)^-$ is negative, and the one of $c_{\alpha,\varepsilon}$ is uniformly bounded by \eqref{eqn:bound_on_errors_c}, which implies that $\sup_{\alpha_1} \alpha_1\left(\hat{\mu}_{\alpha,\varepsilon}(i) -\mu_{\alpha,\varepsilon}(i)\right)^+ = \infty$. This means that we can assume without loss of generality that
	\begin{equation} \label{eqn:comparison_jump_assumption_measures}
		\alpha_1\left(\hat{\mu}_{\alpha,\varepsilon}(i) - \mu_{\alpha,\varepsilon}(i)\right) \rightarrow \infty, \qquad \hat{\mu}_{\alpha,\varepsilon}(i) > \mu_{\alpha,\varepsilon}(i). 
	\end{equation}
	The bound on the right, combined with property (a) of $v$ being a proper kernel, implies that $v(i,j,\hat{\mu}_{\alpha,\varepsilon}) > 0$. We rewrite the term $(a,b) = (i,j)$ in equation \eqref{eqn:proof_comparison_Potts_sum} as
	\begin{multline*}
		\left[\frac{v(i,j,\mu_{\alpha,\varepsilon})}{v(i,j,\hat{\mu}_{\alpha,\varepsilon})} - 1 \right] \times \\
		v(i,j,\hat{\mu}_{\alpha,\varepsilon}) \left[e^{\alpha_1\left(\left(\mu_{\alpha,\varepsilon}(j) - \hat{\mu}_{\alpha,\varepsilon}(j)\right)^- - \left(\mu_{\alpha,\varepsilon}(i) - \hat{\mu}_{\alpha,\varepsilon}(i)\right)^-\right) + c_{\alpha,\varepsilon}(i,j)} - 1\right].
	\end{multline*}
	The term on the second line is bounded above by $M$ introduced in \eqref{eqn:proof_comparison_jump_sup_bound_on_exp} and bounded below by $-\vn{v}$. Thus, we can take a subsequence of $\alpha_1$, also denoted by $\alpha_1$, such that the right-hand side converges. By \eqref{eqn:comparison_jump_assumption_measures}, the right-hand side is non-negative. Therefore, it suffices to show that
	\begin{equation*}
		\liminf_{\alpha_1 \rightarrow \infty} \frac{v(i,j,\mu_{\alpha,\varepsilon})}{v(i,j,\hat{\mu}_{\alpha,\varepsilon})} \leq 1.
	\end{equation*}
	By property (b) of $v$ being a proper kernel, we find $v(i,j,\mu) = v_\dagger(i,j,\mu(i))v_\ddagger(i,j,\mu)$ which implies that
	\begin{multline*}
		\liminf_{\alpha_1 \rightarrow \infty} \frac{v(i,j,\mu_{\alpha,\varepsilon})}{v(i,j,\hat{\mu}_{\alpha,\varepsilon})} = \liminf_{\alpha_1 \rightarrow \infty} \frac{v_\dagger(i,j,\mu_{\alpha,\varepsilon}(i))}{v_\dagger(i,j,\hat{\mu}_{\alpha,\varepsilon}(i))}\frac{v_\ddagger(i,j,\mu_{\alpha,\varepsilon})}{v_\ddagger(i,j,\hat{\mu}_{\alpha,\varepsilon})} \\
		\leq \left(\limsup_{\alpha_1 \rightarrow \infty} \frac{v_\dagger(i,j,\mu_{\alpha,\varepsilon}(i))}{v_\dagger(i,j,\hat{\mu}_{\alpha,\varepsilon}(i))}\right) \left( \lim_{\alpha_1  \rightarrow \infty}\frac{v_\ddagger(i,j,\mu_{\alpha,\varepsilon})}{v_\ddagger(i,j,\hat{\mu}_{\alpha,\varepsilon})} \right) \leq \frac{v_\ddagger(i,j,\mu_{\alpha_2,\varepsilon})}{v_\ddagger(i,j,\mu_{\alpha_2,\varepsilon})} = 1,
	\end{multline*} 
	where we use that $r \mapsto v_{\dagger}(i,j,r)$ is increasing and the bound in \eqref{eqn:comparison_jump_assumption_measures} for the first term and that $\pi \mapsto v_\ddagger(i,j,\mu)$ is continuous and bounded away from zero in a neighborhood of $\mu_{\alpha_2,\varepsilon}$ for the second term.
	
	Thus, cases 1 and 2 inductively establish an upper bound for \eqref{eqn:to_prove_fundamental_bound}, concluding the proof.
\end{proof}

\begin{remark} \label{remark:choice_of_penalization}
	Note that the motivation for the definition of the non-standard $\Psi_1$ in \cite{Kr16b}, as well as the introduction of the use of two penalization functions in this paper comes from the bound obtained in \eqref{eqn:proof_comparison_jump_sup_bound_on_exp}. Indeed, in \cite{Kr16b} the use of $\Psi_1$ allowed us to obtain \eqref{eqn:comparison_jump_assumption_measures}, which is needed to complete the argument.
	
	In our setting, where we work with fluxes, using a single penalization function $\Psi = \Psi_1 + \Psi_2$ multiplied by $\alpha$, would not allow us to obtain \eqref{eqn:comparison_jump_assumption_measures} due to the interference coming from $\Psi_2$. Instead considering these two functions separately with separate multiplicative constants, allows us to establish the important inequality in \eqref{eqn:comparison_jump_assumption_measures}.
	
	\smallskip
	
	Note that an argument like the one carried out in this proof does not seem directly applicable in the context of mass-action kinetics. In particular, if one allows transitions leading to a jumps in the rescaled dynamics of the type $n^{-1} \left(\delta_{c} + \delta_{d} - \delta_{a} - \delta_{b}\right)$, one gets instead of  \eqref{eqn:comparison_jump_assumption_measures} a statement of the type
	\begin{equation*}
		\alpha_1(\hat{\mu}_{\alpha,\varepsilon}(a) - \mu_{\alpha,\varepsilon}(a)) + \alpha_1(\hat{\mu}_{\alpha,\varepsilon}(b) - \mu_{\alpha,\varepsilon}(b)) \rightarrow \infty.
	\end{equation*}
	From such a statement, one cannot derive that $\hat{\mu}_{\alpha,\varepsilon}(a) > \mu_{\alpha,\varepsilon}(a)$ and $\hat{\mu}_{\alpha,\varepsilon}(a) > \mu_{\alpha,\varepsilon}(b)$ for large $\alpha$. This makes it impossible to proceed with the present argument. It seems that a new type of penalization procedure is needed.
\end{remark}

\section{Convergence of operators, exponential tightness and  the variational representation of the rate function} \label{section:proof_ldp_explicit_other_steps}

\subsection{Convergence of operators} \label{section:convergence_of_operators}

We proceed with the verification of \ref{item:LDP_abstract_convergenceHn} of Theorem \ref{theorem:abstract_LDP}, namely that there exists an operator $H$ such that $H \subseteq ex-\LIM H_n$.

\begin{proposition} \label{proposition:convergence_of_operators}
	Consider the setting of Theorem \ref{theorem:ldp_mean_field_jump_process_periodic}.	Let  $H$ be the operator with domain $\cD(H) = C_b^2(E)$ satisfying $Hf(x) = \cH(x,\nabla f(x))$ with $\cH$ as in \eqref{eqn:Hamiltonian}:
	\begin{equation} \label{eqn:Hamiltonian_proof}
		\cH((\mu,w),p) = \sum_{(a,b) \in \Gamma} \mu(a) \overline{r}(a,b,\mu)\left[\exp\left\{p_b - p _a + p_{(a,b)} \right\} - 1 \right].
	\end{equation}
	Then we have $H \subseteq ex-\LIM H_n$.
\end{proposition}

To recall the definition of $ex-\LIM$: we will prove that for each $f \in C_b^2(E)$ there are $f_n \in C_b(\bR^+ \times E_n)$ such that 
\begin{align}
	& \LIM f_n = f, \label{eqn:convergence_functions_time_periodic} \\
	& \LIM H_n f_n = Hf. \label{eqn:convergence_Hfunctions_time_periodic}
\end{align}

The proof of this proposition will be carried out in three steps. 
\begin{itemize}
	\item In Lemma \ref{lemma:convergence_of_operators_modulo_time_period}, we will establish the convergence of operators in the context where the time dependence is essentially removed by working along a time sequence $s_n = s\gamma_n^{-1}$. In this context we will show that for any function $f$ and small perturbation $h_n$ we have $H_n[\gamma_n^{-1} s] (f + h_n) \rightarrow H_0[s]f$.
	\item In Lemma \ref{lemma:integral_perturbation}, we will introduce a sequence of functions that are periodic over the respective intervals on which the jump-rates $r_n$ oscillate. These functions will satisfy the conditions for the small perturbations $h_n$ of the previous step. 
	\item We prove Proposition \ref{proposition:convergence_of_operators} by showing that the perturbations $h_n$ have the effect that $H_n(f + h_n)$ are constant in time. 
\end{itemize}

We first introduce some auxiliary notation.

Let $H_0[t]$ be the operators with domain $C^2_b(E)$ satisfying for $f \in C_b^2(E)$: $H_0[t]f(x) = \cH_0[t](x,\nabla f(x))$ and where
\begin{equation} 
	\cH_0[t]((\mu,w),p) = \sum_{(a,b) \in \Gamma} \mu(a) r(t,a,b,\mu)\left[\exp\left\{p_b - p _a + p_{(a,b)} \right\} - 1 \right]. \label{eqn:Hamiltonian_proof_periodic_0version}
\end{equation}
Following \eqref{eqn:defH_abstract}, we find that for $f$ with $e^{nf} \in \cD(\vec{A}_n)$ that
\begin{align}
	H_n f(t,\mu,w) & := \frac{1}{n} e^{-nf(t,\mu,w)} \cdot \vec{A}_n e^{nf} (t,\mu,w) \notag \\
	& = \partial_t f(t,\mu,w) + H_n[t] f(t,\mu,w). \label{eqn:separation_H_n}
\end{align}
We similarly define the Hamiltonian $H_n[t]$ at time $t$ in terms of the generator $A_n[t]$ at time $t$. For $f$ such that $e^{nf} \in \cD(A_n[t])$, set
\begin{equation*}
	H_{n}[t] f(t,\mu,w)  := \frac{1}{n} e^{-nf(t,\mu,w)} \cdot A_n[t] e^{nf} (t,\mu,w), \notag \\
	= \frac{1}{n} e^{-nf(t,\mu,w)} \cdot (A_n[t] e^{nf} (t,\cdot,\cdot))(\mu,w).
\end{equation*}
Finally, denote for $(\hat{a},\hat{b}) \in \Gamma$ the measure $\mu^{\hat{a},\hat{b}} = \mu + \frac{1}{n}\left(\delta_{\hat{b}} - \delta_{\hat{a}}\right)$ and flux $w^{\hat{a},\hat{b}} = w + \frac{1}{n} \delta_{(\hat{a},\hat{b})}$.

\begin{lemma} \label{lemma:convergence_of_operators_modulo_time_period}
	Consider the setting of Proposition \ref{proposition:convergence_of_operators}. Let $f \in C_b^2(E)$ and let $h_n : \bR^+ \times E_n \rightarrow \bR$ be functions such that

	\begin{equation} \label{eqn:assumption_hn}
		\lim_{n \rightarrow \infty} \sup_{s \in \bR^+, (\mu,w) \in E_n} \sup_{(\hat{a},\hat{b}): \mu(\hat{a}) > 0}  n \left| h_n(s,\mu^{\hat{a},\hat{b}},w^{\hat{a},\hat{b}}) - h_n(s,\mu,w)  \right| = 0.
	\end{equation}
	We then have that $\sup_n \sup_s \vn{H_n[s] (f + h_n) } < \infty$ and
	\begin{equation} \label{eqn:conv_time_dependent_H}
		\lim_{n \rightarrow \infty}  H_n[\gamma_n^{-1} s](f  + h_n)(s,\mu_n,w_n) = H_0[s] f(\mu,w)
	\end{equation} 
	for any sequence  $(\mu_n,w_n) \in E_n$ such that $(\mu_n,w_n) \rightarrow (\mu,w)$ uniformly in $s \geq 0$. 
\end{lemma}

\begin{proof}
	Fix $f \in C_b^2(E)$ and $h_n$ satisfying \eqref{eqn:assumption_hn}. We will prove that $\sup_n \sup_s \vn{H_n[s](f + h_n)} < \infty$ and that for any sequence $(\mu_n,w_n) \in E_n$ such that $(\mu_n,w_n) \rightarrow (\mu,w)$, we have
	\begin{equation} \label{eqn:convergence_H}
		\lim_{n \rightarrow \infty} H_n[\gamma_n^{-1} s](f + h_n)(t_n,\mu_n,w_n) = H_0[s]f(\mu,w)
	\end{equation}
	uniformly in $s$. We consider the left-hand side:
	\begin{align*}
		& H_n[\gamma_n^{-1} s] (f  + h_n)(\mu_n,w_n) \\
		& = \sum_{(a,b) \in\Gamma} \mu_n(a) r_n(\gamma_n^{-1} s, a,b,\mu_n) \\
		& \hspace{3cm} \times \left[e^{n \left(f(\mu_n^{a,b} , w_n^{a,b} - f(\mu_n,w_n)  + h_n(\mu_n^{a,b} ,w_n^{a,b} ) - h(\mu_n,w_n)\right)} - 1 \right] \\
		& = \sum_{(a,b) \in\Gamma} \mu_n(a) r_n(\gamma_n^{-1} s, a,b,\mu_n) \left[e^{n \left(f(\mu_n^{a,b} , w_n^{a,b} - f(\mu_n,w_n)\right) + o(1)} - 1 \right],
	\end{align*}
	where $o(1)$ is a term that vanishes in $n$ uniformly in all parameters by \eqref{eqn:assumption_hn}. As $f \in C_b(E)$, a first order Taylor expansion of $f$ around $\left(\mu_n,w_n\right)$, using Assumption \ref{assumption:jump_rates} \ref{item:results_convergence} and that $\left(\mu_n,w_n\right) \rightarrow (\mu,w)$, we find indeed that \eqref{eqn:convergence_H} holds uniformly in $s$. Note that the first order expansion of $f$ in the exponent can also be used to establish that $\sup_n \sup_s \vn{H_n[s] (f + h_n)} < \infty$.	
\end{proof}

Next, we take care of the fluctuating rates. Note that to prove $H \subseteq ex-\LIM H_n$, we need to find for each $f \in \cD(H)$ a sequence $f_n \in \cD(H_n)$ such that $\LIM f_n = f$ and $\LIM H_n f_n = Hf$. By definition, we have each $f \in \cD(H)$ that $\LIM f  = f$. This, combined with the argument in the proof of Lemma \ref{lemma:convergence_of_operators_modulo_time_period} would be sufficient to establish the convergence of operators if there were no time periodicity. Our context, however, is more difficult. We will modify the functions $f \in \cD(H_n)$ with a perturbative term, itself oscillating, that will cancel out the oscillatory behaviour of the jump rates.

\begin{lemma} \label{lemma:integral_perturbation}
	Consider the setting of Proposition \ref{proposition:convergence_of_operators}. Fix $f \in C_b^2(E)$. For each $n$, define $F_{f,n} : \bR^+ \times E_n$ as
	\begin{equation*}
		F_{f,n}(t,\mu,w) := \int_0^{t} H_n[s] f(\mu,w) \dd s - \frac{t}{\gamma_n^{-1}T_0} \int_0^{\gamma_n^{-1} T_0} H_n[s] f(\mu,w) \dd s.
	\end{equation*}
	Then
	\begin{enumerate}[(a)]
		\item $F_{f,n}(t + \gamma_n^{-1}T_0, \mu,w) = F_{f,n}(t , \mu,w)$ for all $n$ and $(t,\mu,w)$,
		\item We have $\lim_n \vn{F_{f,n}} = 0$.
		\item The functions $F_{f,n}$ satisfy the following Lipschitz estimate
		\begin{equation} \label{eqn:assumption_Ffn}
			\lim_{n \rightarrow \infty} \sup_{s \in \bR^+, (\mu,w) \in E_n} \sup_{(\hat{a},\hat{b}): \mu(\hat{a}) > 0}  n \left| F_{f,n}(\mu^{\hat{a},\hat{b}},w^{\hat{a},\hat{b}}) - F_{f,n}(\mu,w)  \right| = 0.
		\end{equation}
	\end{enumerate}
\end{lemma}

\begin{proof}
	Property (a) is immediate as $F_n(\gamma_n^{-1} T_0,\mu,w) = 0$ and  the $\gamma_n^{-1}t_0$ periodicity of the jump-rates of Assumption \ref{assumption:jump_rates} \ref{item:results_time_periodic}. Due to this periodicity, it suffices for (b) to establish
	\begin{equation*}
		\sup_{t \in [0,\gamma_n^{-1}T_0], (\mu,w) \in E_n} |F_n(t,\mu,w) | \rightarrow 0.
	\end{equation*}
	By a change of variables $u = \gamma_n s$, we find
	\begin{equation*}
		F_n(t,\mu,w) = \gamma_n^{-1} \int_0^{\gamma_n t} H_n[\gamma_n^{-1} u] f (\mu,w) \dd u - \frac{t}{T_0} \int_0^{T_0} H_n[\gamma_n^{-1} u] f (\mu,w) \dd u.
	\end{equation*}
	Using \eqref{eqn:conv_time_dependent_H} with $h_n = 0$, we can replace, up to uniform errors $c_n$ that satisfy $c_n \rightarrow 0$, the integrands by $H[u]f(\mu,w)$. This implies
	\begin{equation*}
		F_n(t,\mu,w) = c_n + \gamma_n^{-1} \int_0^{\gamma_n t} H[u] f (\mu,w) \dd u - \frac{t}{T_0} \int_0^{T_0} H[u](\mu,w) \dd u
	\end{equation*}
	which gives
	\begin{equation*}
		\sup_{t \in [0,\gamma_n^{-1}T_0], (\mu,w) \in E_n} |F_n(t,\mu,w) | \leq c_n + \gamma_n^{-1} C
	\end{equation*}
	concluding the proof of (b). For the proof of (c), we will first establish that here is a constant $C_{f}$ such that for any $n$, any $(\mu,w) \in E_n$, $(\hat{a},\hat{b}) \in \Gamma$ such that $\mu(\hat{a}) > 0$  and $s \geq 0$ we have
	\begin{equation}\label{eqn:Lipschitz}
		n \left| H_n[s]f\left(\mu + \frac{1}{n}(\delta_{\hat{b}} - \delta_{\hat{a}}), w + \frac{1}{n} \delta_{(\hat{a},\hat{b})}\right) -  H_n[s]f(\mu,w) \right| \leq C_{f}.
	\end{equation}
	Applying the definition of $H_n$ yields
	\begin{align*}
		& H_n[s]f\left(\mu^{\hat{a},\hat{b}},w^{\hat{a},\hat{b}}\right) -  H_n[s]f(\mu,w) \\
		& = \sum_{(a,b)\in \Gamma} \mu^{\hat{a},\hat{b}}(a) r_n(s,a,b,\mu^{\hat{a},\hat{b}}) \left[ e^{\frac{1}{n}\left(f(\mu^{\hat{a},\hat{b}} + \frac{1}{n}( \delta_b - \delta_a) , w^{\hat{a},\hat{b}} + \frac{1}{n}\delta_{(a,b)}) \right)} - 1 \right] \\
		& \hspace{3cm} - \sum_{(a,b)\in \Gamma} \mu(a) r_n(s,a,b,\mu) \left[ e^{\frac{1}{n}\left(f(\mu + \frac{1}{n}( \delta_b - \delta_a), w + \frac{1}{n}\delta_{(a,b)}) \right)} - 1 \right].
	\end{align*}
	Standard arguments for obtaining Lipschitz estimates using Assumption \ref{assumption:jump_rates} \ref{item:results_Lipschitz} and $f \in C_b^2(E)$ yield \eqref{eqn:Lipschitz}. Using \eqref{eqn:Lipschitz} in line 5, we find
	\begin{align*}
		& \sup_{t \geq 0} n \left| F_{f,n}(t,\mu^{\hat{a},\hat{b}},w^{\hat{a},\hat{b}}) - F_{f,n}(t,\mu,w) \right| \\
		& \sup_{t \in[0,\gamma_n^{-1}T_0]} n \left| F_{f,n}(t,\mu^{\hat{a},\hat{b}},w^{\hat{a},\hat{b}}) - F_{f,n}(t,\mu,w) \right| \\
		& \quad  = \sup_{t \in[0,\gamma_n^{-1}T_0]} n \left| \int_0^t H_n[s]f(\mu^{\hat{a},\hat{b}},w^{\hat{a},\hat{b}}) -  H_n[s]f(\mu,w) \, \dd s \right| \\
		& \quad \leq \sup_{t \in[0,\gamma_n^{-1}T_0]}n  \int_0^t \left| H_n[s]f(\mu^{\hat{a},\hat{b}},w^{\hat{a},\hat{b}}) -  H_n[s]f(\mu,w) \right| \, \dd s \\
		& \quad \leq \sup_{t \in[0,\gamma_n^{-1}T_0]} \int_0^t C_{f} \dd s \\
		& \quad \leq \gamma_n^{-1} \hat{C}_{f}
	\end{align*}
	establishing (c).
\end{proof}

\begin{proof}[Proof of Proposition \ref{proposition:convergence_of_operators}]
	Recall that for each $f \in \cD(H) = C_b^2(E)$, we need to establish the existence of $f_n \in C_b(\bR^+ \times E_n)$ such that 
	\begin{align}
		& \LIM f_n = f, \label{eqn:convergence_functions_time_periodic_proof} \\
		& \LIM H_n f_n = Hf. \label{eqn:convergence_Hfunctions_time_periodic_proof}
	\end{align}
	Fix $f \in C_b^2(E)$. Using the functions $F_{f,n}$ from Lemma \ref{lemma:integral_perturbation}, we define a suitable collection of functions $f_n$ that approximate $f$ and which take care of the periodic behaviour in the Hamiltonian:
	\begin{align*}
		f_n(t,\mu,w) & := f(\mu,w) -  \left( \int_0^{t} H_n[s] f(\mu,w) \dd s - \frac{t}{\gamma_n^{-1}T_0} \int_0^{\gamma_n^{-1} T_0} H_n[s] f(\mu,w) \dd s \right) \\
		& \, = f(\mu,w) - F_{f,n}(t,\mu,w).
	\end{align*}
	By Lemma \ref{lemma:integral_perturbation}, we have \eqref{eqn:convergence_functions_time_periodic_proof}. We proceed with establishing \eqref{eqn:convergence_Hfunctions_time_periodic_proof}. We use the form in \eqref{eqn:separation_H_n} to establish this result.
	
	Let $(t_n,\mu_n,w_n) \in \bR^+ \times E_n$ be such that $\mu_n,w_n \rightarrow (\mu,w)$ and $\sup_n t_n < \infty$. Note that the application of $H_n[t_n]$ to $f  - F_{f,n}$ and the application of the time derivative to the first integral term of $-F_{f,n}$ yield 
	\begin{equation*}
		H_n[t_n](f - F_{f,n})(t_n,\mu_n,w_n) - H_n[t_n]f(\mu_n,w_n)
	\end{equation*}	
	which converges to $0$ by Lemma \ref{lemma:convergence_of_operators_modulo_time_period} as \eqref{eqn:assumption_Ffn} implies \eqref{eqn:assumption_hn}. We thus obtain the final expression
	\begin{align}
		H_n(f_n  - F_{f,n})(t_n,\mu_n,w_n) & = \frac{1}{\gamma_n T_0} \int_0^{\gamma_n^{-1} T_0} H_n[s] r(s,\mu_n,w_n) \, \dd s + o(1), \notag \\
		& = \frac{1}{T_0} \int_0^{T_0} H_n[\gamma_n^{-1} u] f(\mu_n,w_n) \, \dd u + o(1), \label{eqn:Hn_to_f_with_perturbation}
	\end{align}
	which does not depend on $t_n$. Using \eqref{eqn:conv_time_dependent_H} of Lemma \ref{lemma:convergence_of_operators_modulo_time_period} and the dominated convergence theorem, this yields 
	\begin{align*}
		\lim_n H_n(f_n - F_{f,n})(t_n,\mu_n,w_n) = Hf(\mu,w)
	\end{align*}
	establishing \eqref{eqn:convergence_Hfunctions_time_periodic_proof}.
	
\end{proof}

\subsection{Verifying exponential tightness}

The next condition in Theorem \ref{theorem:abstract_LDP}  is exponential tightness.

\begin{proposition} \label{proposition:exponential_tightness}
	The processes $t \mapsto (t_0 + t, \mu_n(t),w_n(t))$ started in a compact set are exponentially tight on $D_E(\bR^+)$.
\end{proposition}

It is well known in the context of weak convergence that tightness follows from compact containment and the convergence of generators. The same holds in the context of large deviations. The following proposition is the exponential compact containment condition. This property, combined with the convergence of operators established in Proposition \ref{proposition:convergence_of_operators} yields the result by \cite[Corollary 4.19]{FK06} or \cite[Proposition 7.12]{Kr19c}.

\begin{proposition} \label{proposition:compact_containment}
	For each compact set $K \subseteq (\bR^+)^\Gamma$, $T_0 >0$, $T > 0$ and $a > 0$, there is a compact set $\hat{K} \subseteq (\bR^+)^\Gamma$ depending on $K,T,a$ such that
	\begin{equation*}
		\limsup_{n \rightarrow \infty} \sup_{(t_0,\mu_0,w_0): t_0 \leq T_0, w \in K} \frac{1}{n} \log \PR\left[w_n(t) \notin \hat{K} \, \middle| \, (t(0), \mu_n(0),w_n(0)) = (t_0,\mu_0,w_0) \right] \leq -a.
	\end{equation*}
\end{proposition}

The proof is based on a standard martingale argument and is given for completeness, see e.g. Section 4.6 of \cite{FK06}.

\begin{proof}
	Recall that containment function $\Upsilon(\mu,w) = \Upsilon(w) = \sum_{(a,b) \in \Gamma} \log (1 + w_{(a,b)})$ introduced in Lemma \ref{lemma:explicit_good_containment_function} and that the argument in its proof also yields
	\begin{equation*}
		\sup_{\mu,w} \sup_{t}\cH[t]((\mu,w),\nabla \Upsilon(\mu,w)) =: c_{0,\Upsilon} < \infty.
	\end{equation*}
	Choose $\beta > 0$ such that $Tc_{0,\Upsilon} + 1 - \beta \leq -a$. As $\Upsilon$ has compact sublevel sets, we can choose a $c$ such that
	\begin{equation*}
		K \subseteq \left\{(\mu,w) \, \middle| \, \Upsilon(\mu,w) \leq c \right\}
	\end{equation*}
	Next, set $G := \left\{w \, \middle| \, \Upsilon(w) < c + \beta \right\}$ and let $\hat{K}$ be the closure of $G$ (which is a compact set). Let $f(x) := \iota \circ \Upsilon $ where $\iota$ is some smooth increasing function such that
	\begin{equation*}
		\iota(r) = \begin{cases}
			r & \text{if } r \leq c + \beta, \\
			c+ \beta + 1 & \text{if } r > c+ \beta + 2.
		\end{cases}
	\end{equation*}
	Set $g_n := H_n f$. By definition it follows that $\LIM f = f$. We now bound $H_n f$ from above using that $f$ is derived from $\Upsilon$.

	By \eqref{eqn:separation_H_n}, we find
	\begin{align*}
		\sup_{\mu,w,t} H_n f(t,\mu,w) &  = \sup_{\mu,w,t} H_n[t] f(t,\mu,w) \\
		& = \sup_{\mu,w,t} H_n[t] f(\gamma_n^{-1} t,\mu,w).
	\end{align*}
	Noting that $g(\mu,w) = \cH(\mu,w, \nabla \Upsilon(\mu,w)) \leq c_{0,\Upsilon}$ if $w \in \hat{K}$, we find by Lemma \ref{lemma:convergence_of_operators_modulo_time_period} that
	\begin{equation} \label{eqn:upper_bound_HUpsilon_limsup}
		\limsup_n \sup_{t,\mu,w} H_n f(t,\mu,w) \leq \sup_{t,\mu,w} H_0[t] f(\mu,w) \leq c_{0,\Upsilon}.
	\end{equation} 
	We now define a martingale that we will use to control the probability of leaving the set $G$. let
	\begin{equation*}
		M_n(t) := \exp \left\{n \left( f(\mu_n(t),w_n(t))) - f(\mu_n(0),w_n(0)) - \int_0^t g_n(s,\vec{X}_n(s),W_n(s))\dd s \right) \right\}.
	\end{equation*}
	Let $\tau$ be the stopping time $\tau := \inf \left\{t \geq 0 \, \middle| \, w_n(t) \notin  G) \right\}$.
	
	By construction $M_n$ is a martingale and by the optional stopping theorem $t \mapsto M_n(t \wedge \tau)$ is a martingale also. We obtain that if the process is started at $(t_0,\mu_0,w_0)$ such that $w_0 \in K$:
	\begin{align*}
		& \PR\left[w_n(t) \notin  \hat{K} \text{ for some } t \in [0,T]\right] \\
		& \leq \PR\left[w_n(t) \notin  G \text{ for some } t \in [0,T]\right] \\
		& = \bE\left[\bONE_{\left\{w_n(t) \notin G \text{ for some } t \in [0,T]\right\}} M_n(t \wedge \tau) M_n(t\wedge \tau)^{-1} \right] \\
		& \leq \exp\left\{- n \left(\inf_{w \notin  G} \Upsilon(w) - \sup_{w \in K} \Upsilon(w) \right. \right. \\
		& \hspace{4cm} \left. \left. - T \sup_{(\mu,w) \in \cP_n(1,\dots,q) \times G} \sup_t g_n(t,\mu,w) \right) \right\} \\
		& \hspace{2.5cm} \times  \bE\left[\bONE_{\{w_n(t) \notin G \text{ for some } t \in [0,T]\}} M_n(t \wedge \tau) \right].
	\end{align*}
	Using \eqref{eqn:upper_bound_HUpsilon_limsup}, we obtain that the term in the exponential is bounded by $ n\left(c_{0,\Upsilon} T - \beta \right) \leq -n a$ for sufficiently large $n$. The final expectation is bounded by $1$ due to the martingale property of $M_n(t \wedge \tau)$. This establishes the claim.

\end{proof}

\subsection{Establishing the Lagrangian form of the rate function} \label{section:Lagrangian_form}

\begin{proposition} \label{proposition:variational_representation}
	The rate function of Theorem \ref{theorem:abstract_LDP} can be re-expressed in variational form as in Theorem \ref{theorem:ldp_mean_field_jump_process_periodic}.
\end{proposition}

\begin{proof}
	The result follows from a combination of the outcomes of Theorem \ref{theorem:abstract_LDP} with Theorem \cite[Theorem 8.27]{FK06} and \cite[Theorem 8.14]{FK06}. We argue in three steps.
	
	\begin{enumerate}[(1)]
		\item We come up with a new solution $\bfR(\lambda)h$ to the Hamilton-Jacobi equation $f - \lambda Hf = h$ and a new semigroup $\bfV(t)$ using control methods.
		\item Using uniqueness of solutions, we infer that the new solution must equal $R(\lambda)h$ from Theorem \ref{theorem:abstract_LDP}. Similarly we find that the new semigroup $\bfV(t)$ must equal $V(t)$. This leads to a new representation of the rate-function in terms of a Lagrangian $\widehat{\cL}$ given by the Legendre transform of $\cH$.
		\item We show that $\widehat{\cL} = \cL$.
	\end{enumerate}
	
	\smallskip
	
	\textit{Step 1:} We start with the application of \cite[Theorem 8.27]{FK06}. We use this result taking $\bfH = \bfH_\dagger = \bfH_\ddagger$ (in the terminology of \cite{FK06}) all equal to the the operator $H$ of our paper defined in \eqref{eqn:Hamiltonian_proof}. We furthermore use
	\begin{equation} \label{eqn:Lagrangian_abstract}
		\widehat{\cL}((\mu,w),(\dot{\mu},\dot{w})) = \sup_{p} \left\{ \sum_a p_a \dot{\mu}_a + \sum_{(a,b) \in \Gamma} p_{(a,b)} \dot{w}_{(a,b)} - \cH((\mu,w),p) \right\},
	\end{equation}
	and
	\begin{equation*}
		\cA f((\mu,w),(\dot{\mu},\dot{w})) \sum_a \nabla_a f(\mu,w) \cdot \dot{\mu}_a + \sum_{(a,b) \in \Gamma} \nabla_{(a,b)}f(\mu,w) \cdot \dot{w}_{(a,b)}.
	\end{equation*}
	Note that by convex-duality (with respect to the velocity-momentum variables) $\cH$ is the Legendre transform of $\widehat{\cL}$. The conditions for  \cite[Theorem 8.27]{FK06} are Conditions 8.9, 8.10 and 8.11 in \cite{FK06}, which can be checked in a straightforward way, following the methods of e.g. \cite{Kr16b,CoKr17}, or Section 10.3.5 of \cite{FK06} with $\psi = 1$. The final condition for \cite[Theorem 8.27]{FK06} is the comparison principle, which is the result of Theorem \ref{theorem:comparison_principle}.
	
	We obtain from \cite[Theorem 8.27]{FK06}  that there are two families of operators $\bfR(\lambda)$, $\lambda > 0$ and $\bfV(t)$, $t \geq 0$ given in variational form 
	\begin{align*}
		\bfR(\lambda)h(x) & := \sup_{\gamma \in \cA\cC(E), \gamma(0) = x} \int_0^\infty \lambda^{-1} e^{-\lambda^{-1}t} \left(h(\gamma(t)) - \int_0^t \widehat{\cL}(\gamma(s),\dot{\gamma}(s)) \dd s \right) \dd t, \\
		\bfV(t)f(x) &  := \sup_{\gamma \in \cA\cC(E), \gamma(0) = x} f(\gamma(t)) - \int_0^t \widehat{\cL}(\gamma(s),\dot{\gamma}(s)) \dd s,
	\end{align*}
	where $x = (\mu,w)$. Similarly as in Theorem \ref{theorem:abstract_LDP} the results of \cite[Theorem 8.27 and Section 8]{FK06} yield
	\begin{equation} \label{eqn:convergence_semigroup2}
		\bfV(t)f(x) = \lim_{m \rightarrow \infty} \bfR^m\left(\frac{t}{m}\right) f(x).
	\end{equation}
	and such that for $\lambda > 0$ and $h \in C_b(E)$, the function $\bfR(\lambda)h$ is the unique function that is a viscosity solution to $f - \lambda H f = h$. 
	
	\smallskip
	
	\textit{Step 2:} We rewrite the rate function in Lagrangian form.
	
	As both $R(\lambda)h$ and $\bfR(\lambda)h$ are viscosity solutions to $f - \lambda Hf = h$, the comparison principle of Theorem \ref{theorem:comparison_principle} yields that they are equal. By \eqref{eqn:convergence_semigroup} and \eqref{eqn:convergence_semigroup2} we also find $\bfV(t) = V(t)$.
	
	By a duality argument, performed in e.g. \cite[Theorem 8.14]{FK06} it follows that the rate function in Theorem \ref{theorem:abstract_LDP} can be rewritten in Lagrangian form, with Lagrangian given in \eqref{eqn:Lagrangian_abstract}. 
	
	\smallskip
	
	\textit{Step 3:} Finally, we show that $\widehat{\cL} = \cL$. Note that
	\begin{align*}
		& \sum_a p_a \dot{\mu}_a + \sum_{(a,b) \in \Gamma} p_{(a,b)}  \dot{w}_{(a,b)} \\
		& \qquad = \sum_{a} p_a \left(\dot{\mu}_a - \sum_{b: (a,b) \in \Gamma} \left(\dot{w}_{(b,a)} - \dot{w}_{(a,b)}\right)\right) + \sum_{(a,b) \in \Gamma} \dot{w}_{(a,b)}\left(p_{(a,b)} - p_a + p_b\right).
	\end{align*}
	The map $\cH$ only depends on the combinations $p_{(a,b)} - p_a + p_b$. Therefore, taking the Legendre transform of $\cH$, we find that $\widehat{\cL}$ equals infinity if there is some $a$ such that $\dot{\mu}_a \neq \sum_{b: (a,b) \in \Gamma} \left(\dot{w}_{(b,a)} - \dot{w}_{(a,b)}\right)$. In the case that for all $a$ we have $\dot{\mu}_a = \sum_{b: (a,b) \in \Gamma} \left(\dot{w}_{(b,a)} - \dot{w}_{(a,b)}\right)$, the Legendre transform reduces to a supremum over the combinations $p_{(a,b)} - p_a + p_b$. By computing the straightforward Legendre transform of the function $r \mapsto a \left[e^r -1\right]$ with $a > 0$ we find that indeed $\cL = \widehat{\cL}$.
	
	Thus, in both cases $\cL = \widehat{\cL}$, establishing the result of Theorem \ref{theorem:ldp_mean_field_jump_process_periodic}
\end{proof}

\appendix

\section{Viscosity solutions, auxiliary arguments} \label{appendix:viscosity_solutions}

In Section \ref{section:general_method_for_comparison_principle}, we refer at two points to results from \cite{CoKr17}. We repeat these arguments here for sake completeness. The setting is as in Section \ref{section:general_method_for_comparison_principle} .

We start by establishing that we can replace our Hamiltonian $H$ by a proper upper bound $H_\dagger$ and lower bound $H_\ddagger$.

\begin{definition}
	We say that $H_\dagger : \cD(H_\dagger) \subseteq C(E) \rightarrow C(E)$ is a \textit{viscosity sub-extension} of $H$ if $H \subseteq H_\dagger$ (as a graph) and if for every $\lambda >0$ and $h \in C_b(E)$ a viscosity subsolution to $f-\lambda H f = h$ is also a viscosity subsolution to $f - \lambda H_\dagger f =h$. Similarly, we define a \textit{viscosity super-extension}.
\end{definition}

The $H_\dagger,H_\ddagger$ that we will consider are constructed by introducing the unbounded containment function $\Upsilon$ into the domain:
\begin{align*}
	\cD(H_\dagger) & := C^1_b(E) \cup \left\{x \mapsto  (1-\varepsilon)\Psi_\alpha(x,y) + \varepsilon \Upsilon(x) +c \, \middle| \, \alpha,\varepsilon > 0, c \in \bR \right\}, \\
	\cD(H_\ddagger) & := C^1_b(E) \cup \left\{y \mapsto - (1+\varepsilon)\Psi_\alpha(x,y) - \varepsilon \Upsilon(y) +c \, \middle| \, \alpha,\varepsilon > 0, c \in \bR \right\}.
\end{align*}
Here we write $\Psi_\alpha$ for the function $\alpha_1 \Psi_1 + \alpha_2 \Psi_2$. The introduction of the containment function in the domain will allow us to work on compact sets rather than on the full space.

For $f \in \cD(H_\dagger)$, set $H_\dagger f(x) = \cH(x,\nabla f(x))$ and for $f \in \cD(H_\ddagger)$, set $H_\ddagger f(x) = \cH(x,\nabla f(x))$.

\begin{lemma} \label{lemma:viscosity_extension}
	The operator $(H_\dagger,\cD(H_\dagger))$ is a viscosity sub-extension of $H$ and $(H_\ddagger,\cD(H_\ddagger))$ is a viscosity super-extension of $H$.
\end{lemma}

In the proof we need Lemma 7.7 from \cite{FK06}. We recall it here for the sake of readability.
\begin{lemma}[Lemma 7.7 in \cite{FK06}] \label{lemma:extension_lemma_7.7inFK}
	Let $H$ and $H_\dagger : \cD(H_\dagger) \subseteq C(E) \rightarrow  C(E)$ be two operators. Suppose that for all $(f,g) \in H_\dagger$ there exist $\{(f_n,g_n)\} \subseteq H_\dagger$ that satisfy the following conditions:
	\begin{enumerate}[(a)]
		\item For all $n$, the function $f_n$ is lower semi-continuous.
		\item For all $n$, we have $f_n \leq f_{n+1}$ and $f_n \rightarrow f$ point-wise.
		\item Suppose $x_n \in E$ is a sequence such that $\sup_n f_n(x_n) < \infty$ and $\inf_n g_n(x_n) > - \infty$, then $\{x_n\}_{n \geq 1}$ is relatively compact and if a subsequence $x_{n(k)}$ converges to $x \in E$, then
		\begin{equation*}
			\limsup_{k \rightarrow \infty} g_{n(k)}(x_{n(k)}) \leq g(x).
		\end{equation*}
	\end{enumerate}
	Then $H_\dagger$ is a viscosity sub-extension of $H$.\\
	An analogous result holds for super-extensions $H_{\ddagger}$ by taking $f_n$ a decreasing sequence of upper semi-continuous functions and by replacing requirement (c) with
	\begin{enumerate}
		\item[(c$^{\prime}$)] Suppose $x_n \in E$ is a sequence such that $\inf_n f_n(x_n) > - \infty$ and $\sup_n g_n(x_n) <  \infty$, then $\{x_n\}_{n \geq 1}$ is relatively compact and if a subsequence $x_{n(k)}$ converges to $x \in E$, then
		\begin{equation*}
			\liminf_{k \rightarrow \infty} g_{n(k)}(x_{n(k)}) \geq g(x).
		\end{equation*}
	\end{enumerate} 
\end{lemma}

\begin{proof}[Proof of Lemma \ref{lemma:viscosity_extension}]
	We only prove the sub-extension part. 
	
	Consider a collection of smooth functions $\phi_n : \bR \rightarrow \bR$ defined as $\phi_n(x) = x$ if $x \leq n$ and $\phi_n(x) = n+1$ for $x \geq n+1$. Note that $\phi_{n + 1} \geq \phi_n$ for all $n$.
	
	\smallskip
	
	Fix a function $f \in \cD(H_\dagger)$ of the type $f(x) = (1-\varepsilon)\Psi_\alpha(x,y)+\varepsilon \Upsilon(x) + c$ and write $g = H_\dagger f$. Moreover set $f_n := \phi_n \circ f$. Since $f$ is bounded from below, $f_n \in C_c^2(E)$ for all $n$ and as $n \mapsto \phi_n$ is increasing also $n \mapsto f_n$ is increasing and $\lim_n f_n = f$ point-wise.
	
	As $f_n \in C_c^2(E)$, we have $f_n \in \cD(H)$ and we can write $g_n = H f_n$.
	
	We verify conditions of Lemma \ref{lemma:extension_lemma_7.7inFK} for $(f_n,g_n)$ and $(f,g)$. (a) and (b) have already been verified above. For (c), let $\{x_n\}_{n \geq 1}$ be a sequence such that $\sup_n f_n(x_n) = M < \infty$. It follows by the compactness of the level sets of $\Upsilon$ and the positivity of $\Psi_\alpha$ that the sequence $\{x_n\}_{n \geq 1}$ is contained in the compact set 
	\begin{equation*}
		K := \{z \in E \, | \, f(z) \leq M+1\}.
	\end{equation*}
	Note that $K$ has non-empty interior by the assumptions on $\Psi_\alpha$ and $\Upsilon$. In particular, if $h_1,h_2$ are continuously differentiable and if $h_1(z) = h_2(z)$ for $z \in K$, then $\nabla h_1(z) = \nabla h_2(z)$ for $z \in K$.
	
	Suppose $x_{n(k)}$ is a subsequence converging to some point $x$. As $f$ is bounded on $K$,  there exists a sufficiently large $N$ such that for all $n \geq N$ and $y \in K$, we have $f_n(y) = f(y)$ and 
	\begin{equation*}
		g_n(y) = \cH(y,\nabla f_n(y)) = \cH(y,\nabla f(y)) = g(y).
	\end{equation*}
	Thus, we have $\limsup_{k} g_{n(k)}(x_{n(k)}) \leq g(x)$.
\end{proof}

We proceed with a standard argument that is needed for the proof of Proposition \ref{proposition:comparison_conditions_on_H}. It is a copy of the argument of Proposition A.9 of \cite{CoKr17}.

\begin{lemma} \label{lemma:establish_first_bound_viscosity_argument}
	Consider the setting of Proposition \ref{proposition:comparison_conditions_on_H}. Then it holds that 
	\begin{align*} 
		& \sup_x u(x) - v(x) \nonumber\\
		& \leq \frac{h(x_{\alpha,\varepsilon})}{1 - \varepsilon} - \frac{h(y_{\alpha,\varepsilon})}{1+\varepsilon}  \\
		& \qquad + \frac{\varepsilon}{1-\varepsilon}\cH(x_{\alpha,\varepsilon}, \nabla \Upsilon(x_{\alpha,\varepsilon})) + \frac{\varepsilon}{1+\varepsilon}\cH(y_{\alpha,\varepsilon}, \nabla\Upsilon(y_{\alpha,\varepsilon})) \\
		& \qquad +  \lambda \left[\cH\left(x_{\alpha,\varepsilon}, \sum_{i=1}^2 \alpha_i\nabla\Psi_i(\cdot,y_{\alpha,\varepsilon})(x_{\alpha,\varepsilon})\right) - \cH\left(y_{\alpha,\varepsilon},- \sum_{i=1}^2 \alpha_i\nabla\Psi_i(x_{\alpha,\varepsilon},\cdot)(y_{\alpha,\varepsilon})\right)\right]  
	\end{align*}
\end{lemma}

\begin{proof}
	For sake of readability, we write $\Psi_\alpha = \alpha_1 \Psi_1 + \alpha_2 \Psi_2$.

	By Lemma \ref{lemma:viscosity_extension} we get immediately that $u$ is a subsolution to $f - \lambda H_\dagger f = h$ and $v$ is a supersolution to $f - \lambda H_\ddagger f = h$ . Thus, it suffices to verify the comparison principle for the equations involving the extensions $H_\dagger$ and $H_\ddagger$.
	
	\smallskip
	
	By Remark \ref{remark:existence of optimizers}, we can find $x_{\alpha,\varepsilon},y_{\alpha,\varepsilon} \in E$ such that  \eqref{eqn:comparison_principle_proof_choice_of_sequences} is satisfied and such that
	\begin{align}
		& u(x_{\alpha,\varepsilon}) - \lambda \cH\left(x_{\alpha,\varepsilon}, (1-\varepsilon)\nabla \Psi_\alpha(\cdot,y_{\alpha,\varepsilon})(x_{\alpha,\varepsilon}) + \varepsilon \nabla \Upsilon(x_{\alpha,\varepsilon})\right) \leq h(x_{\alpha,\varepsilon}), \label{eqn:ineq_comp_proof_1}\\
		& v(y_{\alpha,\varepsilon}) - \lambda \cH\left(y_{\alpha,\varepsilon},-(1+\varepsilon)\nabla \Psi_\alpha(x_{\alpha,\varepsilon},\cdot)(y_{\alpha,\varepsilon}) - \varepsilon \nabla \Upsilon(y_{\alpha,\varepsilon})\right) \geq h(y_{\alpha,\varepsilon}).\label{eqn:ineq_comp_proof_2}
	\end{align}
	For all $\alpha$ we have
	\begin{align}
		& \sup_x u(x) - v(x) \notag\\
		& = \lim_{\varepsilon \rightarrow 0} \sup_x \frac{u(x)}{1-\varepsilon} - \frac{v(x)}{1+\varepsilon} \notag\\
		& \leq \liminf_{\varepsilon \rightarrow 0} \sup_{x,y} \frac{u(x)}{1-\varepsilon} - \frac{v(y)}{1+\varepsilon} -  \Psi_\alpha(x,y) - \frac{\varepsilon}{1-\varepsilon} \Upsilon(x) - \frac{\varepsilon}{1+\varepsilon}\Upsilon(y) \notag\\
		& = \liminf_{\varepsilon \rightarrow 0} \frac{u(x_{\alpha,\varepsilon})}{1-\varepsilon} - \frac{v(y_{\alpha,\varepsilon})}{1+\varepsilon} - \Psi_\alpha(x_{\alpha,\varepsilon},y_{\alpha,\varepsilon}) - \frac{\varepsilon}{1-\varepsilon}\Upsilon(x_{\alpha,\varepsilon}) -\frac{\varepsilon}{1+\varepsilon}\Upsilon(y_{\alpha,\varepsilon}) \notag \\
		& \leq \liminf_{\varepsilon \rightarrow 0} \frac{u(x_{\alpha,\varepsilon})}{1-\varepsilon} - \frac{v(y_{\alpha,\varepsilon})}{1+\varepsilon}, \label{eqn:basic_inequality_on_sub_super_sol}
	\end{align}
	as $\Upsilon$ and $\Psi_\alpha$ are non-negative functions. We now aim to use that $u$ and $v$ are viscosity sub- and supersolutions. For all $z \in E$, the map $p \mapsto \cH(z,p)$ is convex. Thus, \eqref{eqn:ineq_comp_proof_1} implies that
	\begin{multline} \label{eqn:ineq_comp_proof_1_convexity}
		u(x_{\alpha,\varepsilon}) \leq h(x_{\alpha,\varepsilon}) + (1-\varepsilon) \lambda H(x_{\alpha,\varepsilon}, \nabla \Psi_\alpha(\cdot,y_{\alpha,\varepsilon})(x_{\alpha,\varepsilon}))  \\
		+ \varepsilon \lambda \cH(x_{\alpha,\varepsilon},\nabla \Upsilon(x_{\alpha,\varepsilon})).
	\end{multline}
	We aim for a complementary inequality for $v$. First note that because $\Psi_1,\Psi_2$ are such that $- ( \nabla \Psi_\alpha(x_{\alpha,\varepsilon},\cdot))(y_{\alpha,\varepsilon}) = \nabla \Psi_\alpha(\cdot, y_{\alpha,\varepsilon})(x_{\alpha,\varepsilon})$. Next, we need a more sophisticated bound using the convexity of $\cH$:
	\begin{align*}
		& H(y_{\alpha,\varepsilon}, \nabla \Psi_\alpha(\cdot, y_{\alpha,\varepsilon})(x_{\alpha,\varepsilon})) \\
		& \leq \frac{1}{1+\varepsilon} H(y_{\alpha,\varepsilon},(1+\varepsilon)\nabla \Psi_\alpha(\cdot, y_{\alpha,\varepsilon})(x_{\alpha,\varepsilon}) - \varepsilon \nabla \Upsilon(y_{\alpha,\varepsilon})) + \frac{\varepsilon}{1+\varepsilon} H(y_{\alpha,\varepsilon}, \nabla \Upsilon(y_{\alpha,\varepsilon})).
	\end{align*}
	Thus, \eqref{eqn:ineq_comp_proof_2} gives us
	\begin{equation} \label{eqn:ineq_comp_proof_2_convexity}
		v(y_{\alpha,\varepsilon}) \geq h(y_{\alpha,\varepsilon}) + \lambda (1+\varepsilon) H(y_{\alpha,\varepsilon},\nabla\Psi_\alpha(\cdot,y_{\alpha,\varepsilon})(x_{\alpha,\varepsilon})) - \varepsilon \lambda H(y_{\alpha,\varepsilon},\nabla \Upsilon(y_{\alpha,\varepsilon})).
	\end{equation}
	By combining \eqref{eqn:basic_inequality_on_sub_super_sol} with \eqref{eqn:ineq_comp_proof_1_convexity} and \eqref{eqn:ineq_comp_proof_2_convexity}, we find
	\begin{align*} 
		& \sup_x u(x) - v(x) \nonumber\\
		& \leq \liminf_{\varepsilon \rightarrow 0} \liminf_{\alpha \rightarrow \infty} \left\{ \frac{h(x_{\alpha,\varepsilon})}{1 - \varepsilon} - \frac{h(y_{\alpha,\varepsilon})}{1+\varepsilon} \right.  \\
		& \qquad + \frac{\varepsilon}{1-\varepsilon}\cH(x_{\alpha,\varepsilon}, \nabla \Upsilon(x_{\alpha,\varepsilon})) + \frac{\varepsilon}{1+\varepsilon}\cH(y_{\alpha,\varepsilon}, \nabla\Upsilon(y_{\alpha,\varepsilon})) \\
		& \left. \qquad +  \lambda \left[\cH(x_{\alpha,\varepsilon},\nabla\Psi_\alpha(\cdot,y_{\alpha,\varepsilon})(x_{\alpha,\varepsilon})) - \cH(y_{\alpha,\varepsilon},\nabla\Psi_\alpha(\cdot,y_{\alpha,\varepsilon})(x_{\alpha,\varepsilon}))\right] \vphantom{\sum} \right\}.
	\end{align*}
	This establishes the claim.
\end{proof}

\smallskip

\textbf{Acknowledgement} 

The author thanks Michiel Renger for suggesting to consider flux-large deviation problem from the Hamilton-Jacobi point of view. The author also thanks an anonymous referee for comments and suggestions that improved the exposition of the results.

The author was supported by the Deutsche Forschungsgemeinschaft (DFG) via RTG 2131 High-dimensional Phenomena in Probability – Fluctuations and Discontinuity.

\bibliographystyle{abbrv}
\bibliography{../KraaijBib}
\end{document}